\documentclass[11pt,english]{article}
\usepackage[T1]{fontenc}

\usepackage{graphicx, subfigure}
\expandafter\def\csname ver@subfig.sty\endcsname{}
\usepackage[latin9]{inputenc}
\usepackage{geometry}
\usepackage{amsmath}
\usepackage{appendix}
\usepackage{amssymb}
\usepackage{esint}
\usepackage{makecell,multirow,diagbox}
\usepackage{mathtools}
\usepackage{epstopdf}
\usepackage[latin9]{inputenc}
\usepackage{pict2e}
\usepackage{stmaryrd}
\usepackage{esint}
\usepackage[all,cmtip]{xy}
\usepackage{mathtools}
\usepackage{float}
\usepackage{setspace}
\usepackage{authblk}
\usepackage{amsthm}
\usepackage{mathrsfs}
\usepackage{stmaryrd}
\usepackage{bbold}
\usepackage{color}
\usepackage{tikz}
\usepackage{verbatim}
\usetikzlibrary{arrows,backgrounds,snakes,shapes}
\usepackage{yfonts}
\usepackage[all,cmtip]{xy}
\DeclareMathAlphabet\mathbfcal{OMS}{cmsy}{b}{n}
\usepackage{bbm}
\usepackage{algorithm}
\usepackage{algorithmic}

\theoremstyle{plain}\newtheorem{theorem}{Theorem}[section]
\theoremstyle{plain}\newtheorem{lemma}{Lemma}[section]
\newtheorem{pro}{Proposition}

\newtheorem{remark}{Remark}
\usepackage{color}
\definecolor{marin}{rgb} {0., 0.3, 0.7}
\definecolor{rouge}{rgb} {0.8, 0., 0.}
\definecolor{sepia}{rgb} {0.8, 0.5, 0.}
\usepackage[colorlinks,citecolor=sepia,linkcolor=marin,bookmarksopen,bookmarksnumbered]{hyperref}

\theoremstyle{definition}

\DeclareSymbolFont{largesymbol}{OMX}{yhex}{m}{n}
\DeclareMathAccent{\Widehat}{\mathord}{largesymbol}{"62}

\makeatletter

\usepackage{babel}
\usepackage{subfig}
\graphicspath{{Figure/}}

\makeatother
\DeclareMathSizes{14}{14}{9.8}{7}

\usepackage{lipsum}

\begin{document}
	
	\title{\textbf{Exponential DG methods for Vlasov equations}}
	\date{}
	\author[1]{\small \textbf{Nicolas Crouseilles}} 
	\author[2]{\small \textbf{Xue Hong}} 
	\affil[1] {Universit\'e de Rennes, Inria Rennes (Mingus team) and IRMAR UMR CNRS 6625, F-35042 Rennes, France \& ENS Rennes}
	\affil[2]{Universit\'e de Rennes, Inria Rennes (Mingus team) and IRMAR UMR CNRS 6625, F-35042 Rennes, France}
	
	\maketitle
	\begin{abstract}
	In this work, an exponential Discontinuous Galerkin (DG) method is proposed to solve numerically Vlasov type equations. The DG method is used for space discretization which is combined exponential Lawson Runge-Kutta method for time discretization to get high order accuracy in time and space. In addition to get high order accuracy in time, the use of Lawson methods enables to overcome the stringent condition on the time step induced by the linear part of the system. Moreover, it can be proved that a discrete Poisson equation is preserved. Numerical results on Vlasov-Poisson and  Vlasov Maxwell equations are presented to illustrate the good behavior of the exponential DG method. 
	\end{abstract}
	\setcounter{tocdepth}{1} 
	
	\tableofcontents
	\section{Introduction} 
In this work, we are interested in the numerical simulation of Vlasov type equations using Eulerian based methods. Numerical approximation of Vlasov equations has been the subject of a lot works since these models are widely used to describe the dynamics of charged particles in a plasma through a distribution function $f(t, x, v)$ with the time $t\geq 0$, $x$ the spatial variables and $v$ the velocity variable. Hence, there have been a lot of numerical methods that have been proposed to numerically solve Vlasov equations. The so-called  
Particle-In-Cell(PIC) methods \cite{verboncoeur2005particle,kraus2017gempic} in which the unknown 
is approximated by a sum of Dirac masses with a position and velocity (macro-particles) that solves a differential system. Even if these methods are  
 efficient in high dimensions since only a spatial grid is required, they however suffer from some numerical noise which make them hardly get  an accurate approximation. Indeed, the error slowly decreases when the number of macro-particles increases, which turns out to be 
 a drawback in low density plasma region. On the other side, another family of methods have been developed which uses a grid of the phase space $(x,v)$ like spectral methods \cite{klimas1994splitting,filbet2003comparison}  or finite differences/volumes methods \cite{banks2019high,banks2010new}. These  methods enable to get high order accurate approximation and as such,  can capture fine physical phenomena like Landau damping or filamentation in Vlasov equations. 
 
 However, due to the phase-space grid, these methods are quite costly both in terms of memory and CPU point of view, in particular 
 when high dimensions are considered. Moreover, their stability is controlled by the so-called CFL condition which imposes a constraint on the time step depending of the phase space mesh refinement, which makes them very costly in practice. To overcome this drawback, 
 semi-Lagrangian have been developed \cite{qiu2010conservative,rossmanith2011positivity,einkemmer2019performance} or arbitrary Lagrangian-Eulerian methods  \cite{celia1990eulerian,wang1999family,cai2021eulerian,hong2022generalized}, which allow extra large time stepping sizes with stability by tracking solutions along their characteristics. For high dimensional problems, since (high order) interpolation techniques are required 
 which leads to huge computational cost. Moreover, reaching high order accuracy in time are quite complicated. One possibility is to use 
 splitting methods which enable to deal with simple subproblems that can even be solved exactly. However, the number of stages required to 
 get high order in time become prohibitive (see \cite{crestetto2022comparison}). 
 
For Vlasov type equations, the linear part
induces the most stringent CFL condition since the electromagnetic fields (which induce the nonlinear part) are typically one order of magnitude smaller than the one of the linear advection part. Based on this observation, exponential time integrator have been proposed 
in which the linear part is solved exactly, and as such do not suffer from the stability condition induced by the linear part, 
whereas the nonlinear part 
is solved explicitely. These methods are very popular in a number of applications (\cite{hochbruck2010exponential} and references therein)
and enables to derive easily high order methods in time since they are 
often based on a high order Runge-Kutta method. 
Regarding the use of such time integrators for Vlasov equations, 
we can quote \cite{crouseilles2020exponential, crestetto2022comparison, boutin2022modified}, but these works are based on Fourier techniques 
in space to approximate the linear part, and despite its simplicity  
and its spectral accuracy, Fourier methods are quite limited in terms 
of applications (cartesian domains) 
and suffer from Gibbs phenomenon when non periodic boundary conditions 
are considered. 

In the present work, we focus on exponential type method combined with Discontinuous-Galerkin (DG) method in space to approximate Vlasov type equations. The DG method is a class of finite element methods, in which the approximation space contains completely discontinuous, piecewise polynomials or other basis functions. High order accuracy 
can be obtained and complex geometries with boundary conditions 
can be handled (see DG review article \cite{cockburn2001runge}) 
which is important for physically relevant problems; moreover, one element only communicates with its immediate 
neighbors which is very important for parallelization capability \cite{biswas1994parallel} but also, thanks to this local data structure, the matrices used in these methods are sparse 
which is an important point if one wants to combine DG methods in space 
with exponential methods in time. 

Then, after a finite differences approximation of the Vlasov equation, 
a DG method is employed for the space approximation using central fluxes. Indeed, this choice is motivated by the fact that the DG matrix 
has a pure imaginary spectrum which is not the case when monotone fluxes are considered. 
Moreover, using central fluxes makes the treatment more easier compared to upwind flux. Indeed, the latter requires to split the flux of the linear part into two parts according to the sign of velocity $v$, 
which prevents a discrete Poisson equation from being satisfied because 
of a lack consistency in the nonlinear part. 
Once the semi-discrete equations is obtained, an ODE system has to be solved in time. To do so, exponential time integrators are used to overcome the stringent condition coming from the linear part, but the exponential of a large matrix has to be computed. Thanks to a one side coupling between the distribution function and the electromagnetic fields, it is possible to 
compute explicitely the exponential of the matrix and to derive an efficient numerical scheme which is high order in time, space and velocity, preserves the total mass and a discrete Poisson equation. Some numerical illustrates the good behavior of the method. In particular, the expected order are recovered on a two-dimensional 
linear advection and a good agreement is obtained when 
we compare the DG exponential method to the Fourier exponential  
method. 



The rest of the paper is organized as follows. In Section \ref{subsection:DG}, we present the exponential DG method for one-dimensional (1D) linear transport problems. Section \ref{sec:exponentialDGforVlasov} is dedicated to the construction of the exponential DG method for Vlasov equations includes Vlasov-Amp\`ere (1dx-1dv) and Vlasov Maxwell (1dx-2dv) equations. 
In Section \ref{sec:DG2Dx}, we discuss some extensions of the 
exponential method to high dimensional Vlasov-Maxwell equations. 
In Section \ref{sec:numerialtests}, the capability of the proposed exponential DG method is illustrated through several numerical tests. Finally, after some concluding remarks, several appendices 
details some specific aspects of the method. 
 
	\section{Exponential DG method for 1D transport equation}
 \label{subsection:DG}
	
	We firstly consider the 1D transport equation:
	\begin{equation}
 \label{linear_transport}
 \left\{
 \begin{aligned}
     &u_t+au_x = 0, \quad x\in[x_a, x_b],\\
     &u(0,x)=u_0(x).
 \end{aligned}
 \right.
	\end{equation}
	For simplicity, we assume periodic boundary conditions, and the velocity field $a$ is a constant. Here we take $a$ as 1 for simplicity. We perform a partition of the computational domain $x_a=x_{\frac12}< x_{\frac32}<\cdots< x_{N+\frac12} =x_b$ as the mesh partition.
	Let $I_j=[x_{j-\frac12}, x_{j+\frac12} ]$ denote an element of length $\Delta x_j=x_{j+\frac12}-x_{j-\frac12}$ and define $\Delta x=\max_{j}\Delta x_j.$ For simplicity, we consider the uniform mesh in this paper with $\Delta x_j=\Delta x=(x_b-x_a)/N$.
	We define the finite dimensional approximation space, $V_h^k = \{ v_h:  v_h|_{I_j} \in P^k(I_j) \}$, where $P^k(I_j)$ denotes the set of polynomials of degree at most $k$ on $I_j$. 
 For any $\psi \in V_h^k$, we also denote the left limit of $\psi$ at cell boundary as $\psi^-$ and the right limit as $\psi^+$.
	Multiply \eqref{linear_transport} by the test function $\psi \in V_h^k$, integrate on cell $I_j$ and integrate by parts, we end up with the semi-discrete DG scheme: 
 find $u_h \in V_h^k$ such that
	\begin{equation}
 \label{eq:semi-discreteDGschemefortransport}
		\int_{{I}_j}(\partial_t u_h\psi)dx =-  \hat{F}|_{{x}_{j+\frac12} }  \psi^-|_{{x}_{j+\frac12} } +   \hat{F}|_{{x}_{j-\frac12} } \psi^+|_{{x}_{j-\frac12} }   + \int_{{I}_j}F\psi_xdx, \ j=1,...,N,
	\end{equation}
	where $F(u) \doteq  u$ and $\hat{F}$ is chosen as either a central or upwind fluxes
 \begin{equation}
 \label{choice_fluxes}
 \mbox{central flux: } \; \hat{F}|_{{x}_{j\pm \frac12} }=\frac{ u^- +u^+}{2}|_{{x}_{j\pm \frac12}}, \;\;  \mbox{upwind flux: } \; \hat{F}|_{{x}_{j\pm \frac12} }= u^-|_{{x}_{j\pm \frac12}},
 \end{equation}
%
Next, we consider $\xi_j^m \; (m=0,1,\dots, k)$ a basis of $P^k(I_j)$ 
and we choose a modal basis defined as $\xi^m_j(x)=((x-x_j)/\Delta x)^m$ so that we have the representation $u_h(t,x)|_{I_j}=\sum_{m=0}^k u_j^m(t)\xi_j^m(x)$, with $u_j^m(t)$ the degree of freedom. The semi-discrete DG scheme can eventually be written as an ordinary differential equation (ODE) satisfied by the DG degrees of freedom $u^m_j(t)$ for $m=0, \dots, k$ and $j=1, \dots, N$. 
Introducing the vector $\mathbf{u}(t)\in\mathbb{R}^{(k+1)N}$ 
\begin{equation}
\label{repres_u}
\mathbf{u}(t)=(u_1^0, u_1^1, \dots,u_1^k, \; u_2^0, u_2^1, \dots,u_2^k,\dots,u_N^0, u_N^1, \dots,u_N^k)^T(t), 
\end{equation}
the semi-discretized problem simply becomes  
\begin{equation}
\label{semi_discrete}
\frac{d \mathbf{u}}{dt}  =A\mathbf{u}, \mbox{ with } A\in \mathbb{M}_{{(k+1)N,(k+1)N}}(\mathbb{R}).    
\end{equation}
The 'DG-matrix' $A$ contains the DG approximation 
\eqref{eq:semi-discreteDGschemefortransport} for which the details are given in Appendix \ref{appendixB1}. From this semi-discrete in space formulation, a 
Runge-Kutta discretization is classically used to get high order accuracy in time \cite{cockburn2001runge,zhang2004error,zhang2010stability}. But it seems also natural to use an exponential method for time discretization, which turns out to be exact in this simple linear transport case. Denoting $\mathbf{u}^n\approx \mathbf{u}(t^n)$ the fully 
discretized unknown (with $t^n=n\Delta t, n\in \mathbb{N}$ and $\Delta t>0$ the time step), the exponential-DG scheme thus writes 
\begin{equation}
\label{exp-DG-linear}
\mathbf{u}^{n+1} = \exp(A\Delta t)\mathbf{u}^n, \;\;\; \forall n\in \mathbb{N}, 
\end{equation} 
with  $\mathbf{u}^0=\mathbf{u}_0$ 
($\mathbf{u}_0$ being the degrees of freedom of the initial condition $u_0$ in \eqref{linear_transport}). 
Note that from $\mathbf{u}^n\in \mathbb{R}^{(k+1)N}$ (whose components are denoted by $(u_j^m)^n, j=1, \dots, N$ and $m=0, \dots, k$, following \eqref{repres_u}), it is possible to reconstruct a piecewise polynomial function $u^n_h\in V_h^k$ from 
\begin{equation}
\label{DG-reconstruct}
u_h^n(x) = \sum_{m=0}^k (u_j^m)^n \xi_j^m(x), \;\; \forall x\in I_j, \;\; j=1, \dots, N.  
\end{equation}


The properties of such an approximation obviously depends  on the structure of the DG-matrix $A$ which is discussed now. From the calculations (given in Appendix \ref{appendixB1}), the DG-matrix $A$ obtained with the central flux and periodic boundary conditions enjoys a circulant tri-diagonal block structure so that it can be written as:
\begin{equation}
\label{dg_matrix}
A=\frac{1}{\Delta x}\left(\begin{array}{lllll}
	C_1& C_2 & \bf{0} & \hdots &C_3 \\ 
	C_3 & C_1& C_2 & \bf{0}&\hdots  \\
    \bf{0}& \ddots & \ddots & \ddots &C_2 \\
    C_2 & \bf{0}& \hdots & C_3 & C_1
\end{array}
\right), \;\; \mbox{ with } C_j=M^{-1}D_j \; (j=1,2,3)
\end{equation}
where the matrix elements of the matrices  $M, D_j \in \mathbb{M}_{{(k+1),(k+1)}}(\mathbb{R})$ are given by (for $\ell, m=1, \dots, k+1$, see Appendix \ref{appendixB1} for details) 
\begin{eqnarray} 
M_{\ell,m}=\frac{(1/2)^{m+\ell-1}}{m+\ell-1}[1-(-1)^{m+\ell-1}], \!\!\!\!\!\! && (D_2)_{\ell,m}= (-1)^{m}(1/2)^{m+\ell-1},\nonumber\\
\label{expression_MD}
(D_1)_{\ell,m}=(1/2)^{m+\ell-2}\Big(\frac{\ell-1}{m+\ell-2}-\frac{1}{2}\Big)[1-(-1)^{m+\ell-2}], \!\!\!\!\!\!  &&   (D_3)_{\ell,m}=(-1)^{\ell-1}(1/2)^{m+\ell-1},  
\end{eqnarray}
with $(D_1)_{1,1}=0$ by convention. 
The choice of central fluxes implies the matrix $A$ is diagonalizable and the eigenvalues are pure imaginary. This has been checked numerically and some discussions are performed in the following remarks.

\begin{remark}
In \cite{tee2007eigenvectors}, the author   proposes a way to deduce the eigenvalues of $A\in \mathbb{M}_{{(k+1)N,(k+1)N}}(\mathbb{R})$ from the eigenvalues of some matrices of size $(k+1)$, which can be computed explicitly for small values of $k$ (numerically for larger $k$). Considering $(\rho_j)_{j=0, \dots, N-1}$ the $N$-th roots of the unity  ($\rho_j^N=1$ for $j=0, \dots, N-1$), the $(k+1)N$ eigenvalues of $A$ given by \eqref{dg_matrix} can be deduced from the $(k+1)$ eigenvalues of ${\cal C}_j=C_1+\rho_j C_2 + \rho_j^{N-1} C_3$ for $j=0, \dots, N-1$, where ${\cal C}_j\in\mathbb{M}_{{(k+1),(k+1)}}$. 
Then, we checked numerically that the eigenvalues of ${\cal C}_j$ are pure imaginary for all $j=0, \dots, N-1$, and we deduce from \cite{tee2007eigenvectors} 
that it is also true for the eigenvalues  of $A$. 
\end{remark}

\begin{remark}
\label{remark_eigenA}
We explore another way to check the eigenvalues of $A$ given by \eqref{dg_matrix} are pure imaginary by using symbolic software. Denoting $P_A(\lambda)$ the characteristic polynomial of $A$, we made the following observations   
\begin{itemize}
\item odd case: $(k+1)N=2d+1$.  In this case, we have $P_A(\lambda)= \lambda \sum_{\ell=0}^d a_{2\ell}\lambda^{2\ell}$ with $a_{2\ell}\in \mathbb{R}$ and the roots can be written as $\left\{0, \lambda_j, \bar{\lambda}_j, j=1, \dots, d\right\}$, in particular $0$ is a simple eigenvalue in this case. Since $P_A(-\lambda)=P_A(\lambda)$, we deduce  Re$(\lambda_j)=0$. 
\item even case: $(k+1)N=2d$.  In this case, we have $P(\lambda)= \lambda^2 \sum_{\ell=0}^{d-1} a_{2\ell}\lambda^{2\ell}$ with $a_{2\ell}\in \mathbb{R}$ and the  roots can be written as $\left\{0, \lambda_j, \bar{\lambda}_j, j=1, \dots, d-1\right\}$, in particular $0$ is a double  eigenvalue in this case. Since $P_A(\lambda)=P_A(-\lambda)$, we deduce Re$(\lambda_j)=0$. 
\end{itemize}
\end{remark}




We can now study the stability  of the numerical scheme. To do so, we write the following proposition.  
\begin{pro}
Let consider the matrix $A\in \mathbb{M}_{(k+1)N,(k+1)N}(\mathbb{R})$ given by \eqref{dg_matrix}-\eqref{expression_MD} There exists $C>0$ such that, for any time $t$ and any $k, N$, we have $\|\exp(At)\| \leq C$, with $\|\cdot \|$ an induced matrix norm.   
\begin{proof}
First, we write $A=\Delta x^{-1} A_1$ (with $\Delta x=(x_b-x_a)/N$) and since $A$ is diagonalizable, there exist $P$ invertible and $D$ diagonal such that $A=\Delta x^{-1}P D P^{-1}$. Let us remark from \eqref{dg_matrix}-\eqref{expression_MD} that $A_1$ does not depend on the space mesh $\Delta x$ (and then dos not depend on $N$), so does the matrix $P$. Thus, there exists $C>0$ (independent of $N$) such that  cond$(P)\equiv\|P\|\|P^{-1}\|\leq C$, where $\|\cdot \|$ denotes an induced matrix norm. 
Now, since the eigenvalues of $A$ are pure imaginary for all $k,N$, we have $D_{j,j} = i\lambda_j, \lambda_j\in \mathbb{R},  j=1,2,\dots,(k+1)N$. Finally, 
we get for all $t$ 
$$
 \|\exp(At)\|=\|P\exp(\Delta x^{-1}D)P^{-1}\| \leq \mbox{cond}(P) \; \|\exp(\Delta x^{-1}D)\| = \mbox{cond}(P) \leq C. 
$$
\end{proof}
\end{pro}

We end this section by proving  an error estimate for exponential DG method.
\begin{pro}
Let $u(t,x)$ the exact solution of 1D transport problem \eqref{linear_transport} with a smooth initial condition $u_0$  and let $u_h^n\in V_h^k$ the numerical solution, $n=0, \dots, N$ (with $N=T/\Delta t$, $T$ being the final time and $\Delta t$ the time step) obtained from \eqref{exp-DG-linear}-\eqref{DG-reconstruct} where $A$ is the DG-matrix given by \eqref{dg_matrix}-\eqref{expression_MD}. Then we have the following $L^2$-norm error estimate:
$$
||u(t^n,\cdot)-u_h^n||_{L^2}\leq C \Delta x^{k}, \;\;\; 
$$

\begin{proof}
First, we introduce $u_h(t,x)$  the exact solution of the semi-discrete DG scheme \eqref{semi_discrete}. The classical DG projection analysis gives $||u(t^n,\cdot)-u_h(t^n,\cdot)||_{L^2}\leq C \Delta x^{k},$ where $C$ is a positive constant independent on $\Delta x$ (see the details in \cite{richter1988optimal,lesaint1974finite,johnson1986analysis,peterson1991note,zhang2004error,zhang2010stability}, but a simple proof of stability and error estimate is given in Appendix \ref{appendixA}).  Since the exponential method exactly solves the semi-discrete DG scheme, we have $||u_h(t^n,\cdot)-u_h^n||_{L^2}=0$. Finally we have $$||u(t^n,\cdot)-u_h^n||_{L^2}\leq||u(t^n,\cdot)-u_h(t^n,\cdot)||_{L^2}+||u_h(t^n,\cdot)-u_h^n||_{L^2}\leq C \Delta x^{k}.$$ 
\end{proof}
\end{pro}

\begin{remark}
It is worthy to be mentioned that the $(k+1)$th order optimal convergence rate have been proved for DG with monotone flux (see \cite{johnson1986analysis,peterson1991note,zhang2004error,zhang2010stability}). However, $k$th order sub-optimal convergence rate is proved for DG with central flux in \cite{liu_shu_DG_convergence} and a discussion is performed 
according to the oddness of $k$. 
\end{remark}
 
\section{Semi-discretization of some Vlasov models with Discontinuous Galerkin  method}\label{sec:exponentialDGforVlasov}
In this section, we consider the numerical approximation of Vlasov-Maxwell equations using DG framework in space (as presented in the previous section) and finite differences in the velocity direction. The semi-discretization (in both space and velocity) is presented and we will see that the so-obtained ODE system is amenable to Lawson time integrators.  
 
 We first present the methodology on the $1d_x-1d_v$ case on the Vlasov-Amp\`ere system and then we consider the $1d_x-2d_v$ case on the Vlasov-Maxwell system.  
 
	\subsection{Vlasov-Amp\`ere equation}
 \label{DG_VA}
	The equation we address is the following Vlasov-Amp\`ere model satisfied by the distribution function $f(t, {x}, {v})\geq 0$ and the electric field $E(t, x)\in \mathbb{R}$ 
 with $t\geq 0, x\in [0, L] \; (L>0)$ and $v\in \mathbb{R}$,
	\begin{equation}
		\label{eq:modelVA}
		\left\{
		\begin{aligned}
			&\frac{\partial f}{\partial t} + {v}  \frac{\partial f}{\partial { x}} +E  \frac{\partial f}{\partial { v}}= 0, \\
			&\partial_t E=-\int_ {\mathbb{R}} vf \mathrm{d}v+\bar{J}, \;\;\;\;  \bar{J}=\frac{1}{L}\int_0^L \int_{\mathbb{R}} v f dx dv,  \\
		\end{aligned}
		\right.
	\end{equation} 
 with the initial conditions $(f_0(x, v), E_0(x))$ such that the Poisson equation is satisfied initially $\partial_x E_0 = \int_{\mathbb{R}} f_0dv - \bar{\rho}$, with $\bar{\rho} =(1/L) \int_0^L \int_{\mathbb{R}} f_0 dv dx$ and periodic boundary conditions are imposed in space. The Vlasov-Amp\`ere system is equivalent the Vlasov-Poisson model where the electric field satisfies the Poisson equation $\partial_x E= \int_{\mathbb{R}} f {d}{ v}-\bar{\rho}$.

 \subsubsection{Semi-discretization}
We shall use a DG method in the space direction $x$ as presented in the previous section and we consider a truncated domain $[-v_{ \max}, v_{ \max}]$ in the velocity direction discretized by $v_{j} = -v_{ \max} + j \Delta v, j=0, \dots, N_v$, $\Delta v=2 v_{ \max}/N_v$ being the velocity mesh step.  We firstly present DG discretization for ${E}$ in the $x$-direction with $I_i=[x_{i-\frac12}, x_{i+\frac12} ],\ i=1,...,N_x$ ($N_x$ being the number of cells):
$$
E(t,x) \approx E_h(t,x)=\sum_{i=1}^{N_x} E_h(t, x)|_{I_i}= \sum_{i=1}^{N_x} \sum_{m=0}^{k}{E}_{i}^m(t) \xi^m(x). 
$$
Then, considering finite difference method for $f$ in $v$ direction, we consider the DG approximation in $x$-direction through  
$$
f(t,x,v_j)\approx f_h(t,x,v_j)= \sum_{i=1}^{N_x} \sum_{m=0}^{k}{f}_{i}^m(t,v_j) \xi^m(x),\ j=0,...,N_v.
$$
As in the previous section, we denote $\mathbf{f}_{j}(t)\in \mathbb{R}^{(k+1)N_x}$ the vector of the DG coefficients $f^m_i(t, v_j)$ of $f_h(t, x, v_j)$ using DG in space and evaluated at the velocity grid $v_j$ whereas $\mathbf{E}\in \mathbb{R}^{(k+1)N_x}$ denotes the vector of DG  coefficients $E^m_i$ of  $E_h(t, x)$.

For Vlasov-Amp\`ere equation \eqref{eq:modelVA}, we have the following DG scheme with the DG representation of $E$ and $f$:
\begin{equation}
\begin{aligned}
&\sum_{m=0}^{k}\left[\Big(\partial_t{f}_{i}^m(t,v_j) \xi^m, \xi^\ell\Big)_{I_i}-v_j\Big({f}_{i}^m(t,v_j) \xi^m,\partial_x\xi^\ell(x)\Big)_{I_i}\right]+
 v_j\left[\{f_h(t, x, v_j)\} \xi^\ell \right]^{i+1/2}_{i-1/2}\\
    &+ \sum_{m=0}^{k}\Big(\sum_{n=0}^{k}E_i^n\xi^n({\mathcal{D}f_{i}^m})(t,v_j) \xi^m, \xi^\ell\Big)_{I_i}=0,
\end{aligned}
\end{equation}
where we used the central flux $\{f_h(t, x, v_j)\}|_{x_{i\pm 1/2}} 
= \frac{1}{2}(f_h(t, x_{i\pm 1/2}^+, v_j) + f_h(t, x_{i\pm 1/2}^-, v_j))$, $\ell=0,1,2,...,k,\ i=1,2,...N_x$ and  ${\mathcal{D}f}(v_j)$ denotes a discrete approximation of $(\partial_v f)(v_j)$ (an example would be $({\mathcal{D}f})(v_j)=\frac{f(v_{j+1})-f(v_{j-1})}{2\Delta v}$ but any  higher order finite difference  approximation can be used).
We denote $$\mathbf{f}_{j,i}(t)=({f}_{i}^0(t,v_j), {f}_i^1(t,v_j), ....,{f}_i^k(t,v_j))^T, \ i=1,2,...,N_x,$$ 
$$\mathbf{E}_i(t)=({E}_{i}^0(t), {E}_i^1(t), ....,{E}_i^k(t))^T, \ i=1,2,...,N_x,$$
and 
$$\mathbf{{\mathcal{D}f}}_{j,i}(t)=(({\mathcal{D}f}_{i}^0)(t,v_j), ({\mathcal{D}f}_i^1)(t,v_j), ....,({\mathcal{D}f}_i^k)(t,v_j))^T, \ i=1,2,...,N_x.$$
We can rewrite the DG discretization as an ODE system of size 
$(k+1)N_x$ for each $j=1, \dots, N_v$ 
\begin{equation}
\label{VA_dg_1}
\begin{aligned}
    \Delta x \left(\begin{array}{llll}
	M&  &  &  \\ 
	 & M&  &  \\
     &  & \ddots & \\
     &  &  & M\\
\end{array}
\right)
\frac{d}{dt}\left(\begin{array}{llll}
	\mathbf{f}_{j,1}\\ 
 \mathbf{f}_{j,2}\\ 
 \vdots\\ 
\mathbf{f}_{j,N_x}
\end{array}\right)
&-
v_j \left(\begin{array}{llll}
	D_1& D_2 & \hdots &D_3 \\ 
	D_3 & D_1& D_2 &  \\
     & \ddots & \ddots &D_2 \\
    D_2 & \hdots & D_3 & D_1\\
\end{array}
\right)
\left(\begin{array}{llll}
	\mathbf{f}_{j,1}\\ 
 \mathbf{f}_{j,2}\\ 
 \vdots\\ 
\mathbf{f}_{j,N_x}
\end{array}\right)\\
&+
\left(\begin{array}{llll}
	B_1&  &  &  \\ 
	 & B_2&  &  \\
     &  & \ddots & \\
     &  &  & B_{N_x}\\
\end{array}
\right)
\left(\begin{array}{llll}
	\mathbf{({\mathcal{D}f})}_{j,1}\\ 
 \mathbf{({\mathcal{D}f})}_{j,2}\\ 
 \vdots\\ 
\mathbf{({\mathcal{D}f})}_{j,N_x}
\end{array}\right)={\bf 0},
\end{aligned}
\end{equation}
where the matrices $M, D_1, D_2, D_3$ are the same as in the previous section (see also Appendix \ref{appendixB1}) 
and $B_i (i=1, \dots, N_x)$ are matrices of size $k+1$ with elements $(B_i)_{\ell,m}=(\sum_{n=0}^{k}E_i^n\xi^n \xi^m, \xi^\ell)_{I_i}$.

Introducing now the following vector containing the degrees of freedom of $f_h$ and $E_h$
\begin{eqnarray*}
\mathbf{f}_j(t)\!\!&\!\!=\!\!&\!\!({f}_{1}^0(t,v_j), \dots,{f}_1^k(t,v_j), {f}_2^0(t,v_j), \dots,{f}_2^k(t,v_j),\dots,{f}_{N_x}^0(t,v_j), \dots,{f}_{N_x}^k(t,v_j))^T,\nonumber\\ 
\mathbf{E}(t)\!\!&\!\!=\!\!&\!\!({E}_{1}^0(t), \dots,{E}_1^k(t), {E}_2^0(t), \dots,{E}_2^k(t),\dots,{E}_{N_x}^0(t),  \dots,{E}_{N_x}^k(t))^T,
\end{eqnarray*}
and $\mathbf{{\mathcal{D}f}}_j$ is defined similarly as $\mathbf{f}_j$, 
we can rewrite the DG scheme \eqref{VA_dg_1} as 
\begin{equation}
\label{eqn:VdiscretizationDG}
    \partial_t \mathbf{f}_{j}=v_{j}A \mathbf{f}_{j}- \tilde{\mathbf{E}} (\mathbf{{\mathcal{D}f}})_j,
\end{equation}
$A\in \mathbb{M}_{(k+1)N_x,(k+1)N_x}$ is the DG-matrix \eqref{dg_matrix} and $\tilde{\mathbf{E}}\in \mathbb{M}_{(k+1)N_x,(k+1)N_x}$ is a block diagonal matrix 
composed of $N_x$ block matrices of size $(k+1)\times (k+1)$ defined by  $(\Delta x M)^{-1}B_i, i=1, \dots, N_x$. Let remark that we consider central finite differences method to approximate  $(\mathbf{{\mathcal{D}f}})_j$ to avoid to discuss the sign of matrix $ \tilde{\mathbf{E}}\in \mathbb{M}_{(k+1)N_x,(k+1)N_x}$ compared with upwind FD method.

Let us now discuss the discretization of the Amp\`ere equation. Since our goal is to find a consistent discretization 
that is compatible with a discrete Poisson equation,  we first discuss how to solve the Poisson equation.  
Using the above discretization, we will use the DG matrix 
$A$ which is an approximation of $(-\partial_x)$. 
Thus, a direct approximation of the initial Poisson  equation $\partial_x E(0,x)= \int_{\mathbb{R}} f_0(x,v) {d}{ v}-\bar{\rho}$ would be $-A\mathbf{E}^0=\sum_{j}\mathbf{f}^0_{j}\Delta v-\bar{\rho}$ 
(with $\mathbf{E}^0$ and $\mathbf{f}^0_{j}$ the degrees of freedom of $E(0, x)$ and $f_0(x, v_j)$. However, as mentioned in Remark \ref{remark_eigenA}, 
$A\in \mathbb{M}_{(k+1)N_x,(k+1)N_x}$ is not invertible and we then introduce $\Pi\in \mathbb{M}_{(k+1)N_x,(k+1)N_x}$ the projection onto the 
Ker$(A)$ so that $(A+\Pi)$ is invertible on $R(A)$ with $R(A)$ the range of $A$. Here we impose condition $\Pi \mathbf{E}^0=0$ to preserve the uniqueness of the solution $\mathbf{E} \in \mathbb{R}^{(k+1)N_x}$, which is similarly as the constraint $\int_{\mathbb{R}} E(x,v) {d}{ x}=0$ for Poisson equation itself.
We then consider the following discretized Poisson equation 
\begin{equation}
\label{discrete_poisson_dg}
-(A+\Pi)\mathbf{E}^0=({\bf{1}}-\Pi)\Big(\sum_{j}\mathbf{f}^0_{j}\Delta v-\bar{\rho}\Big)=({\bf{1}}-\Pi)\sum_{j}\mathbf{f}^0_{j}\Delta v,
\end{equation} 
with ${\bf{1}}$ the identity matrix of size $(k+1)N_x$ and  where in the last equality, we used the fact that constants belong to Ker$(A)$  (see Appendix \ref{appendixC} for details).

We deduce the discretization of Amp\`ere equation from the time derivative of the discretized Poisson equation inspired from \eqref{discrete_poisson_dg}, that is:  
$-(A+\Pi)\mathbf{E}(t)=({\bf{1}}-\Pi)\sum_{j}\mathbf{f}_{j}(t)\Delta v$. 
Considering the time derivative 
of the latter equation and using \eqref{eqn:VdiscretizationDG} leads to 
\begin{eqnarray*}
        -(A+\Pi)\partial_t\mathbf{E}(t)&=&({\bf{1}}-\Pi)\partial_t\sum_{j}\mathbf{f}_{j}(t)\Delta v=({\bf{1}}-\Pi)\sum_{j}v_jA\mathbf{f}_{j}(t)\Delta v\\
        &=&A({\bf{1}}-\Pi)\sum_{j}v_j\mathbf{f}_{j}(t)\Delta v=(A+\Pi)({\bf{1}}-\Pi)\sum_{j}v_j\mathbf{f}_{j}(t)\Delta v, 
\end{eqnarray*}
where we used $\sum_j {\cal D} {\bf f}_j = 0$ and some 
relations between $A$ and $\Pi$. Hence, we consider the following DG discretization of the Amp\`ere equation 
\begin{equation}
\label{ampere_DG}
\partial_t \mathbf{E}(t)=-({\bf{1}}-\Pi)\sum_{j}v_j\mathbf{f}_{j}(t)\Delta v.
\end{equation}

Finally, gathering \eqref{eqn:VdiscretizationDG} and \eqref{ampere_DG} enables to get the following semi-discretized scheme for the Vlasov-Amp\`ere system 
\begin{equation}
	\label{eq:model_VlasovAmpere}
	\left\{
	\begin{aligned}
		&\partial_t \mathbf{f}_{j}=v_{j}A \mathbf{f}_{j}- \tilde{\mathbf{E}} (\mathbf{{\mathcal{D}f}})_j ,  \\
		&\partial_t \mathbf{E}=-\sum_{j} v_{ j} \mathbf{f}_{j} \Delta v + \sum_{j} v_{ j} \Pi\mathbf{f}_{j} \Delta v. 
	\end{aligned}
	\right.
\end{equation} 

In view of the time discretization, we introduce the following vector of semi-discrete unknown $U=(\vec{\mathbf{f}}, \mathbf{E})\in \mathbb{R}^{(k+1)N_x (N_v+1)}$ with 
$\vec{\mathbf{f}}=(\mathbf{f}_1, \mathbf{f}_2, \dots, \mathbf{f}_{N_v})\in \mathbb{R}^{(k+1)N_x N_v}$. Then the previous system \eqref{eq:model_VlasovAmpere} can be rewritten as 
\begin{equation}
	\label{ULU_core}
	\partial_t U =L U + N(U), 
\end{equation}
with $L\in \mathbb{M}_{(k+1)N_x(N_v+1),(k+1)N_x(N_v+1)}$ given by 
\begin{eqnarray}
 \label{L_matrix}
	L&=& \left( 
	\begin{array}{lllllllccc}
		\;\;\;\;\;\;{v_{1}}A & \;\;\;\;{\bf 0}_{\tilde{N},\tilde{N}} & {\bf 0}_{\tilde{N},\tilde{N}} \;\;\;\;\hdots &  \;\;\;\; {\bf 0}_{\tilde{N},\tilde{N}} &  {\bf 0}_{\tilde{N},\tilde{N}} \\
		\;\;\;\;\;\;{\bf 0}_{\tilde{N},\tilde{N}} & \;\;\;\;	{v_{2}}A &  {\bf 0}_{\tilde{N},\tilde{N}} \;\;\;\;\hdots & \;\;\;\;  {\bf 0}_{\tilde{N},\tilde{N}} & \vdots \\
		\;\;\;\;\;\;\vdots & \;\;\;\;\ddots & \ddots  &  \;\;\;\;\vdots  &  \vdots \\
		\;\;\;\;\;\;{\bf 0}_{\tilde{N},\tilde{N}} &\;\;\;\; \hdots  &   {\bf 0}_{\tilde{N},\tilde{N}}& \;\;\;\;{v_{N_v}}A &  {\bf 0}_{\tilde{N},\tilde{N}} \\
		\!\!\!-\Delta v v_1 ({\bf 1}-\Pi)& -\Delta v v_2 ({\bf 1}-\Pi) & \hdots  & -\Delta v v_{N_v} ({\bf 1}-\Pi) & {\bf 0}_{\tilde{N},\tilde{N}}
	\end{array} 
	\right) 
 \end{eqnarray}
 where we denote $\tilde{N}=(k+1)N_x$ and ${\bf 1}={\bf 1}_{\tilde{N},\tilde{N}}$ the identity matrix of size $\tilde{N}\times\tilde{N}$. Finally,  $U, N(U)\in \mathbb{R}^{(k+1)N_x(N_v+1)}$ are given by 
\begin{eqnarray*} 
 	U&=& \left( 
	\begin{array}{llll}
		{\mathbf{f}}_{1}\\
		{\mathbf{f}}_{2}\\
		\vdots \\
		\!\!\! {\mathbf{f}}_{N_{v}}  \!\!\!\!\\
		\mathbf{E}
	\end{array} 
	\right), \;\;\;\;\;\;\;\;\;\;\;\; \;\;\;\;\;\;\;\;\;\;\;\; \;\;\;\;\;\;\;\;\;\;\;\;  
	N(U)= 
	\left( 
	\begin{array}{llllccc}
		-\tilde{\mathbf{E}} (\mathbf{{\mathcal{D}f}})_1\\
		-\tilde{\mathbf{E}} (\mathbf{{\mathcal{D}f}})_2\\
		\;\;\;\;\vdots\\
		-\tilde{\mathbf{E}} (\mathbf{{\mathcal{D}f}})_{N_v}\\
		\;\;\;\;{\bf 0} 
	\end{array} 
	\right).  
\end{eqnarray*}

\subsubsection{Time discretisation} 
The goal of this part is to present time discretization of \eqref{ULU_core} to get a fully discretized scheme of the Vlasov-Amp\`ere system \eqref{eq:modelVA}. 
The form \eqref{ULU_core} is amenable to exponential scheme \cite{hochbruck2010exponential,crouseilles2020exponential} which is motivated by the fact that the linear part acts on a different scale compared to the nonlinear part in Vlasov type problems. Moreover, as discussed in Section \ref{subsection:DG}, the linear part can be computed exactly thanks to the exponential. Among the exponential schemes, we shall use the Lawson class of methods for stability reasons \cite{crouseilles2020exponential}. 

Denoting $U^n=({\vec{\bf f}^n, {\bf E}^n})\approx ({\vec{\bf f}(t^n), {\bf E}(t^n)}) = U(t^n)$ 
with $t^n=n\Delta t, \forall n\in \mathbb{N}$ ($\Delta t>0$ being the time step),  
the simplest (first order in time) Lawson method can be written as 
\begin{equation}
\label{eq:Lawson}
U^{n+1} = \exp(\Delta t L)U^n +\Delta t \exp(\Delta t L) N(U^n). 
\end{equation}
High order methods can be obtained from Runge-Kutta methods using the corresponding Butcher tableau \cite{crouseilles2020exponential}. 

The key point of exponential methods lies in the computation of  $\exp(\Delta t L)$ with $L$ given by \eqref{L_matrix}. To compute $\exp(A t)\  \forall t>0$, one considers the linear part only $\partial_t U=LU$. 
First, we observe that the distribution function $\mathbf{f}_{j}$ part is decoupled from the electric field part, so that it can be solved directly and we have, similarly as in Section \ref{subsection:DG} 
\begin{equation}\label{eqn:frepresentationVP}
   \mathbf{f}_{j}(t)=\exp(v_{j}At)\mathbf{f}_{j}(0). 
\end{equation} 
Now, let consider 
the electric field equation 
$$
{\partial_t \mathbf{E}}=
-\sum_{j} v_{ j} ({\bf 1}-\Pi)\mathbf{f}_{j}(t) \Delta v,
$$ 
which gives, after integrating it in time 
$$
\mathbf{E}(t)=\mathbf{E}(0)-\sum_{j} \int_0^t v_{ j} ({\bf 1}-\Pi)\mathbf{f}_{j}(s) \Delta v \mathrm{d}s. 
$$
Now, replacing $\mathbf{f}_{j}(s)$ by $\exp(v_{j}As)\mathbf{f}_{j}(0)$ from \eqref{eqn:frepresentationVP} enables to get an explicit  expression of $\mathbf{E}(t)$. Indeed, using properties used to derive \eqref{discrete_poisson_dg} we have
\begin{eqnarray*}
		\mathbf{E}(t)&=&\mathbf{E}(0)-\sum_{j} \int_0^t v_{ j} ({\bf 1}-\Pi)\mathbf{f}_{j}(s) \Delta v \mathrm{d}s=\mathbf{E}(0)-\sum_{j} \int_0^t v_{ j} ({\bf 1}-\Pi)\exp(v_{j}As)\mathbf{f}_{j}(0)\mathrm{d}s \Delta v \\
        &=&\mathbf{E}(0)-\sum_{j} \int_0^t v_{ j} ({\bf 1}-\Pi)\Bigg[\sum_{k=0}^{\infty}\frac{(v_{j}As)^k}{k!}\Big(\Pi\mathbf{f}_{j}(0)+({\bf 1}-\Pi)\mathbf{f}_{j}(0)\Big)\Bigg]\mathrm{d}s \Delta v\\
       &=&\mathbf{E}(0)-\sum_{j} v_{ j} ({\bf 1}-\Pi)\int_0^t\Bigg[\Pi\mathbf{f}_{j}(0)+\sum_{k=0}^{\infty}\frac{(v_{j}(A+\Pi)s)^k}{k!}({\bf 1}-\Pi)\mathbf{f}_{j}(0)\Bigg]\mathrm{d}s \Delta v\\
        &=&\mathbf{E}(0)-\sum_{j} v_{ j} ({\bf 1}-\Pi)\Bigg[ t\Pi\mathbf{f}_{j}(0)+\sum_{k=0}^{\infty}\frac{(v_{j}(A+\Pi))^kt^{k+1}}{(k+1)!}({\bf 1}-\Pi)\mathbf{f}_{j}(0)\Bigg] \Delta v\\
        &=&\mathbf{E}(0)-\sum_{j}\Bigg[({\bf 1}-\Pi)\sum_{k=0}^{\infty}(A+\Pi)^{-1}\frac{(v_{j}(A+\Pi)t)^{k+1}}{(k+1)!}({\bf 1}-\Pi)\mathbf{f}_{j}(0)\Bigg] \Delta v\\
        &=&\mathbf{E}(0)-\sum_{j}  ({\bf 1}-\Pi)(A+\Pi)^{-1}(\exp(v_{j}At)-{\bf 1})({\bf 1}-\Pi)\mathbf{f}_{j}(0) \Delta v\\
        &=&\mathbf{E}(0)-\sum_{j}  (A+\Pi)^{-1}(\exp(v_{j}At)-{\bf 1})({\bf 1}-\Pi)\mathbf{f}_{j}(0)  \Delta v\\
        &=&\mathbf{E}(0)+\Delta v\sum_{j} (A+\Pi)^{-1}({\bf 1}-\exp(v_{j}At))\mathbf{f}_{j}(0).
\end{eqnarray*}
Hence, denoting $A+\Pi$ as $\tilde{A}$, we deduce from the above calculation the explicit expression of $\exp(L t)$  
\begin{equation}
\label{expL}
\exp(L t) = 
\left( 
\begin{array}{lllllllccc}
	\;\;\;\;\;\;\;  e^{v_{1}A t} & \;\;\;\;\;\;\;{\bf 0}_{\tilde{N},\tilde{N}} & {\bf 0}_{\tilde{N},\tilde{N}} \;\;\;\;\;\;\;\hdots &  \;\;\;\;\;\;\; {\bf 0}_{\tilde{N},\tilde{N}} &  {\bf 0}_{\tilde{N},\tilde{N}} \\
		\;\;\;\;\;\;\;\;  {\bf 0}_{\tilde{N},\tilde{N}} & 	\;\;\;\;\;\;\; e^{v_{2}A t} &  {\bf 0}_{\tilde{N},\tilde{N}} \;\;\;\;\;\;\; \hdots &  \;\;\;\;\;\;\; {\bf 0}_{\tilde{N},\tilde{N}} & \vdots \\
		\;\;\;\;\;\;\;\; \vdots & \;\;\;\;\;\;\;\ddots & \ddots  & \;\;\;\;\;\;\; \vdots  &  \vdots \\
		\;\;\;\;\;\;\;\; {\bf 0}_{\tilde{N},\tilde{N}} & \;\;\;\;\;\;\;\hdots  &   {\bf 0}_{\tilde{N},\tilde{N}}& \;\;\;\;\;\;\; e^{v_{N_v}A t} &  {\bf 0}_{\tilde{N},\tilde{N}} \\
		\!\!\Delta v \tilde{A}^{-1}({\bf 1}-e^{v_{1}A t})& \Delta v  \tilde{A}^{-1}({\bf 1}-e^{v_{2}A\Delta t}) & \hdots  & \Delta v \tilde{A}^{-1}({\bf 1}-e^{v_{N_v}A t}) & {\bf 1}
\end{array} 
\right).
\end{equation}
Hence, the exponential-DG scheme for the Vlasov-Amp\`ere equation corresponds to \eqref{eq:Lawson}-\eqref{expL}. For this scheme, one can prove in the following proposition that a discrete Poisson equation is satisfied for each iteration.

\begin{pro}
The exponential DG method \eqref{eq:Lawson}-\eqref{expL} (and its generalization to high order Lawson Runge-Kutta) 
satisfied by $U^n=(\vec{\bf{f}}^n, {\bf E}^n)$ preserves the following discretized Poisson equation: 
$$
(A+\Pi) \mathbf{E}^n=-\sum_{j}({\bf 1}-\Pi) \mathbf{f}^n_{j}\Delta v, \;\;\; \forall n\in \mathbb{N}^\star, 
$$
provided that it is satisfied at the initial time $n=0$. 
 Here, $A$ is the DG matrix given by \eqref{dg_matrix}-\eqref{expression_MD}, $\Pi$ is the orthogonal projection onto Ker$(A)$  
and ${\bf{1}}$ is the identity matrix of size $(k+1)N_x$, 

\end{pro}

\begin{proof}
We present the proof for first order Lawson case (forward Euler), the proof can be generalized to arbitrary explicit Runge-Kutta scheme. First, we assume 
the Poisson equation 
$(A+\Pi)\mathbf{E}^0=-\sum_{j}({\bf 1}-\Pi)\mathbf{f}_{j}^0\Delta v$ holds at the initial time. \\
Next, from the scheme \eqref{eq:Lawson} with $\exp(Lt)$ 
given by \eqref{expL}, we have 
\begin{equation}
\label{eqn:VADGschemef}
    \mathbf{f}_{j}^{n+1}=\exp(v_{j}A\Delta t)\mathbf{f}_{j}^n-\Delta t\exp(v_{j}A\Delta t)(\tilde{\mathbf{E}} \mathbf{{\mathcal{D}f}})_{j}^n, 
\end{equation}
whereas for the ${\bf E}$ component, we have 
\begin{eqnarray*}
		\mathbf{E}^{n+1}&=&\mathbf{E}^n+\Delta v\sum_{j} (A+\Pi)^{-1}({\bf 1}-\exp(v_{j}A\Delta t))\mathbf{f}_{j}^n \nonumber\\
  &&-\Delta t\Delta v\sum_{j} (A+\Pi)^{-1}({\bf 1}-\exp(v_{j}A\Delta t))(\tilde{\mathbf{E}} \mathbf{{\mathcal{D}f}})_{j}^n\nonumber\\
  &=&\mathbf{E}^n+\Delta v\sum_{j} (A+\Pi)^{-1}({\bf 1}-\Pi)({\bf 1}-\exp(v_{j}A\Delta t))\mathbf{f}_{j}^n \nonumber\\ 
  &&-\Delta t\Delta v\sum_{j} (A+\Pi)^{-1}({\bf 1}-\Pi)({\bf 1}-\exp(v_{j}A\Delta t))(\tilde{\mathbf{E}} \mathbf{{\mathcal{D}f}})_{j}^n. \nonumber\\
 \end{eqnarray*} 
The last term can be split into two parts: the first one 
$-\Delta t\Delta v (A+\Pi)^{-1}({\bf 1}-\Pi)\sum_{j}(\tilde{\mathbf{E}} \mathbf{{\mathcal{D}f}})_{j}^n$ vanishes 
thanks to the conservative properties of  the discrete operator $\mathbf{{\mathcal{D}}}$ whereas the second one, we use  \eqref{eqn:VADGschemef} to get 
 \begin{eqnarray*}
\Delta t\Delta v (A+\Pi)^{-1}\sum_{j}({\bf 1}-\Pi)\exp(v_{j}A\Delta t)(\tilde{\mathbf{E}} \mathbf{{\mathcal{D}f}})_{j}^n \!\!&\!\!=\!\!&\!\! \Delta v (A+\Pi)^{-1}\sum_{j}({\bf 1}-\Pi)(\exp(v_{j}A\Delta t)\mathbf{f}_{j}^n-\mathbf{f}_{j}^{n+1})\nonumber\\
\hspace{-2cm}&\hspace{-5cm}=&\hspace{-2.5cm} \Delta v (A+\Pi)^{-1}\sum_{j}\Big[({\bf 1}-\Pi)\exp(v_{j}A\Delta t)\mathbf{f}_{j}^n-({\bf 1}-\Pi)\mathbf{f}_{j}^{n+1}\Big].
\end{eqnarray*} 
Finally, we then get 
 \begin{eqnarray}
\mathbf{E}^{n+1} &=&\mathbf{E}^n+\Delta v\sum_{j} (A+\Pi)^{-1}({\bf 1}-\Pi)({\bf 1}-\exp(v_{j}A\Delta t))\mathbf{f}_{j}^n\nonumber\\
&&+\Delta v (A+\Pi)^{-1}\sum_{j}\Big[({\bf 1}-\Pi)\exp(v_{j}A\Delta t)\mathbf{f}_{j}^n-({\bf 1}-\Pi)\mathbf{f}_{j}^{n+1}\Big]\nonumber\\
        	\label{eq:Poissoneequationpreserving}
		&=&\mathbf{E}^n+\Delta v\sum_{j} (A+\Pi)^{-1}( ({\bf 1}-\Pi)\mathbf{f}_{j}^n-({\bf 1}-\Pi)\mathbf{f}_{j}^{n+1}).
\end{eqnarray} 
By induction, if the discrete Poisson equation is satisfied at iteration $n$, then it is satisfied at iteration $n+1$ and the proof is complete.

\end{proof}

\begin{remark}
In this remark, we discuss how to compute the matrix $\Pi$. 
As mentioned in Remark \ref{remark_eigenA}, 
when $(k+1)N_x$ is odd (referred as odd case), $0$ is a single eigenvalue and the associated eigenvector $u_1$
corresponds to the constant function in the DG space (see Appendix \ref{appendixC}). When $(k+1)N_x$ is even, $0$ 
has a double multiplicity and one has to find another eigenvector $u_2$ (see Appendix \ref{appendixC}). 
Once we get the eigenvectors $u_1,\ u_2\in \mathbb{R}^{(k+1)N_x}$ of $A$ associated to the eigenvalue $0$, then $\Pi$ is given by $\Pi x=\langle x,u_1\rangle u_1+\langle x,u_2\rangle u_2 = (u_1^T\otimes u_1+u_2^T\otimes u_2)x$, for all $x\in \mathbb{R}^{(k+1)N_x}$, 
with the Kronecker product $\otimes$ and $(v^T\otimes v)_{i,j} = v_i v_j$. Some examples are given in Appendix \ref{appendixC}.
\end{remark}

\subsection{Vlasov Maxwell equations 1dx-2dv}
\label{VM1dx-2dv}
In this part, We consider the following Vlasov-Maxwell 1dx-2dv model satisfied by $f(t, x, v_1, v_2)$, $E_1(t, x)$, $E_2(t, x)$, $B(t, x)$, 
with $t\geq 0, x\in [0,L] (L>0)$ and $(v_1, v_2)\in \mathbb{R}^2$
\begin{equation}
	\label{eq:model}
	\left\{
	\begin{aligned}
		&\partial_t f+v_1\partial_x f+ E_1 \partial_{v_1} f +E_2 \partial_{v_2} f +B(v_2 \partial_{v_1} f -v_1 \partial_{v_2} f )=0, \\
		&\partial_t B =- \partial_x E_2, \\
		&\partial_t E_1=-\int_ {\mathbb{R}^2} v_1f \mathrm{d}v_1\mathrm{d}v_2+\bar{J}_1,\\
		&\partial_t E_2=-\partial_x B-\int_{\mathbb{R}^2} v_2f \mathrm{d}v_1\mathrm{d}v_2+\bar{J}_2,
	\end{aligned}
	\right.
\end{equation} 
with the initial conditions $(f_0(x,v), E_1^0(x), E_2^0(x), B^0(x))$ 
such that the Poisson equation is satisfied initially 
$\partial_x E_1^0(x)= \int_{\mathbb{R}^2} f_0(x,v) dv_1 dv_2- \bar{\rho}$, with $\bar{\rho}=\frac{1}{L}\int_0^L \int_{\mathbb{R}^2} f_0(x,v) dx dv_1 dv_2$ and periodic boundary conditions are imposed in space. Here $\bar{J}_i=\frac{1}{L}\int_0^L \int_{\mathbb{R}^2} v_i f dx dv_1 dv_2$ ensures that the electric fields are zero average in space. 

\subsubsection{Semi-discretization}
We follow the lines of the above subsection and  use a DG method in the space direction $x$ and we consider a grid in the velocity direction $v_{\ell,j_\ell} = -v_{\ell, \max} + j_\ell \Delta v_\ell, v_{\ell}\in [-v_{\ell, \max},v_{\ell, \max}],\ell=1, 2, \Delta v_\ell=2v_{\ell, \max}/{N_{v_\ell}}, N_{v_\ell}\in \mathbb{N}^\star$. 
The definitions of the different objects are a direct extension of 
the definitions introduced in the previous part.  
Indeed, we denote by $\mathbf{f}_{j_1,j_2}\in\mathbb{R}^{(k+1)N_x}$ the DG  coefficient vector of $f_{j_1,j_2}(t, x)=f_h(t, x, v_{j_1}, v_{j_2})\approx f(t, x, v_{j_1}, v_{j_2})$ in space and evaluated at the velocity grid, and $\mathbf{E_1}, \mathbf{E_2}, \mathbf{B}\in\mathbb{R}^{(k+1)N_x}$ are 
the DG  coefficient vectors of 
$(E_{1,h}, E_{2,h}, B_h)(t, x)\approx (E_1, E_2, B)(t, x)$. 
Moreover, $(\tilde{\mathbfcal{F}} {\cal D} \mathbf{f})_{j_1, j_2}$ 
(with $\tilde{\mathbfcal{F}}_{j_1, j_2}\in \mathbb{M}_{(k+1)N_x, (k+1)N_x}(\mathbb{R})$) is obtained as previously by a DG approximation of the nonlinear term $({\cal F} \cdot  \nabla_v f), \ \mbox{ with }  {\cal F} =({E_1}+ {B} v_2, {E_2}  -{B} v_1 )$. For fixed indices 
$j_1, j_2$, the derivation of the numerical scheme is very similar 
to the 1dx-1dv case and we then obtain the following semi-discretized (in space and velocity) scheme for $j_\ell=1, \dots, N_{v_\ell} (\ell=1, 2)$
\begin{equation}
	\label{eq:model_dg_2dv}
	\left\{
	\begin{aligned}
		&\partial_t \mathbf{f}_{j_1,j_2}=v_{1,j_1}A \mathbf{f}_{j_1,j_2}- ( \tilde{\mathbfcal{F}} {\cal D} \mathbf{f})_{j_1,j_2},\\ 
		&\partial_t \mathbf{B} =A \mathbf{E_2}, \\
		&\partial_t \mathbf{E_1}=-\sum_{j_1, j_2} v_{1, j_1} ({\bf 1}-\Pi)\mathbf{f}_{j_1,j_2} \Delta v_1\Delta  v_2,\\
		&\partial_t \mathbf{E_2}=A \mathbf{B}- \sum_{j_1, j_2} v_{2, j_2} ({\bf 1}-\Pi)\mathbf{f}_{j_1,j_2} \Delta v_1\Delta v_2,
	\end{aligned}
	\right.
\end{equation} 
where $A$ is given by \eqref{dg_matrix}-\eqref{expression_MD} and ${\bf 1}$ the identity matrix of size $(k+1)N_x$ and $\Pi$ the projection matrix onto Ker$(A)$ introduced in the previous part. 
Let denote $U(t)=(\vec{\mathbf{f}}(t), \mathbf{B}(t), \mathbf{E_1}(t), \mathbf{E_2}(t))$, where  
$\vec{\mathbf{f}}(t)\in \mathbb{R}^{(k+1)N_x N_{v_1}N_{v_2}}$ contains the DG coefficients \\ 
$\mathbf{f}_{j_1, j_2}$, 
$\mathbf{B}(t),\mathbf{E}_1(t), \mathbf{E}_2(t)\in \mathbb{R}^{(k+1)N_x}$ denote the DG coefficients of the electromagnetic fields.  
Using the above notations and the following ones 
\begin{eqnarray*}
\vec{v_\ell}\in \mathbb{R}^{N_{v_\ell}}, &&(\vec{v_\ell})_{j}=v_{\ell,j}=-v_{\ell, \max}+ j \Delta v_\ell, \;\;  j=1, \dots, N_{v_\ell} \mbox{ and } \ell=1, 2, \nonumber\\ 
\vec{f}_{\star, j_2}\in \mathbb{R}^{(k+1)N_x N_{v_1}} && \!\!\!\!\!\!\!\!\mbox{ for  } j_2=1, \dots, N_{v_2}. 
\end{eqnarray*}
We also introduce the following compact notations for the size of the matrices: $\tilde{N}=(k+1)N_x$ and $\tilde{N}_1= (k+1)N_x N_{v_1} = \tilde{N}N_{v_1}$.  
The system \eqref{eq:model_dg_2dv} can be recast  as 
\begin{equation}
	\label{ULU2_core}
	\partial_t U =L U + N(U), 
\end{equation}
with 
\begin{eqnarray}
	U&=& \left( 
	\begin{array}{llll}
		\vec{\mathbf{f}}_{\star,1}\\
		\vec{\mathbf{f}}_{\star,2}\\
		\vdots \\
		\!\!\! \vec{\mathbf{f}}_{\star,N_{v_2}}  \!\!\!\!\\
		\mathbf{B}\\
		\mathbf{E_2}\\
		\mathbf{E_1}
	\end{array} 
	\right), \;\;\;\;\;\;\;\;\;\;\;\;  
	N(U)= 
	\left( 
	\begin{array}{llllccc}
		-(\tilde{\mathbfcal{F}}{\cal D} \vec{\mathbf{f}})_{\star,1}\\
		-(\tilde{\mathbfcal{F}}{\cal D} \vec{\mathbf{f}})_{\star,2}\\
		\vdots\\
		-(\tilde{\mathbfcal{F}}{\cal D} \vec{\mathbf{f}})_{\star,N_{v_2}}\\
		{\bf 0} \\
		{\bf 0} \\
		{\bf 0}
	\end{array} 
	\right), \;\;\;\;\;\;\;\;\;\;\;\; \vec{\mathbf{f}}_{\star,j_2}\in\mathbb{R}^{\tilde{N_1}}, \forall j_2=1, \dots, N_{v_2}  \nonumber\\ 
 \label{L_matrix2D}
	L&=& \left( 
	\begin{array}{lllllllccc}
		 \small{\mbox{diag}(\vec{v_{1}})\!\otimes \! A}& {\bf 0}_{\tilde{N_1},\tilde{N_1}} & {\bf 0}_{\tilde{N_1},\tilde{N_1}}\hdots &   {\bf 0}_{\tilde{N_1},\tilde{N_1}} &  {\bf 0}_{\tilde{N_1},\tilde{N}} & {\bf 0}_{\tilde{N_1},\tilde{N}} & {\bf 0}_{\tilde{N_1},\tilde{N}}\\
		{\bf 0}_{\tilde{N_1},\tilde{N_1}} 
  & 
     \small{\mbox{diag}(\vec{v_{1}})\!\otimes \! A}
  &  {\bf 0}_{\tilde{N_1},\tilde{N_1}}\hdots &   {\bf 0}_{\tilde{N_1},\tilde{N_1}} & \vdots &  \vdots & \vdots\\
		\vdots & \ddots & \ddots  &  \vdots  &  \vdots & \vdots & \vdots\\
		{\bf 0}_{\tilde{N_1},\tilde{N_1}} & \hdots  &   {\bf 0}_{\tilde{N_1},\tilde{N_1}}& 
     \small{\mbox{diag}(\vec{v_{1}})\!\otimes \! A}
  &  {\bf 0}_{\tilde{N_1},\tilde{N}}  &  {\bf 0}_{\tilde{N_1},\tilde{N}}  &  {\bf 0}_{\tilde{N_1},\tilde{N}} \\
		{\bf 0}_{\tilde{N},\tilde{N_1}} & \hdots  & {\bf 0}_{\tilde{N},\tilde{N_1}} & {\bf 0}_{\tilde{N},\tilde{N_1}} & {\bf 0}_{\tilde{N},\tilde{N}} &  A & {\bf 0}_{\tilde{N},\tilde{N}}  \\
{\cal E}_{2,1} & {\cal E}_{2,2} & \hdots & {\cal E}_{2, N_{v_2}} & A & {\bf 0}_{\tilde{N},\tilde{N}} & {\bf 0}_{\tilde{N},\tilde{N}}\\
{\cal E}_{1} & {\cal E}_{1} & \hdots & {\cal E}_{1} & {\bf 0}_{\tilde{N},\tilde{N}} & {\bf 0}_{\tilde{N},\tilde{N}} & {\bf 0}_{\tilde{N},\tilde{N}}\\
	\end{array} 
	\right) 
\end{eqnarray}  
where diag$(\vec{v_1})\in \mathbb{M}_{N_{v_1},N_{v_1}}(\mathbb{R})$  denotes the diagonal matrix with $\vec{v_1}\in \mathbb{R}^{N_{v_1}}$ on its diagonal, $\otimes$ denotes Kronecker product, 
$\small{\mbox{diag}(\vec{v_{1}})\!\otimes \! A} \in \mathbb{M}_{\tilde{N}_1, \tilde{N}_1}$  and 
${\bf 0}_{m,n}$ is the zero matrix with $m$ lines and $n$ columns. Moreover, the matrices ${\cal E}_{1}, {\cal E}_{2,j_2} \in \mathbb{M}_{\tilde{N}, \tilde{N}_1}(\mathbb{R})$ for $j_2=1, \dots, N_{v_2}$ are defined by 
\begin{eqnarray*}
{\cal E}_{1} &=& -\Delta v_1\Delta v_2 \vec{v_{1}}\otimes ({\bf 1}-\Pi), \nonumber\\
{\cal E}_{2,j_2} &=& -\Delta v_1\Delta v_2 v_{2, j_2}\mathbbm{1}\otimes ({\bf 1}-\Pi) \;\; \mbox{ with } \;\; \mathbbm{1}=(1, \dots, 1)\in\mathbb{R}^{N_{v_1}}. 
\end{eqnarray*}
The size of the matrix $L$ is $(\tilde{N_1} N_{v_2} + 3\tilde{N})\times(\tilde{N_1} N_{v_2} + 3\tilde{N})=(3+N_{v_1}N_{v_2})(k+1)N_x\times(3+N_{v_1}N_{v_2})(k+1)N_x$. Even if it is a large matrix, one can see that $L$ is sparse which will help to compute its exponential. 

\subsubsection{Time discretization}
We now study the time discretization of \eqref{ULU2_core} and as previously, we will design an exponential scheme. To do so, the discrete unknown $U^n\approx U(t^n) \; (t^n=n\Delta t, \Delta t>0)$ is updated by 
\begin{equation}
\label{expo_DG_1}
U^{n+1} = \exp(\Delta t L)U^n +\Delta t \exp(\Delta t L) N(U^n), 
\end{equation}
and one has to compute $\exp(\Delta t L)$. We give the following proposition to show the representation of $\exp(\Delta t L)$.

\begin{pro}
\label{prop_expo_2dv}
The exponential of the matrix $L$ given by \eqref{L_matrix2D} is given by 
$$
\exp(L\Delta t) = 
\left( 
\begin{array}{lllllllccc}
e^{\Delta t\small{\mbox{diag}(\vec{v_{1}})\!\otimes \! A}}& 
{\bf 0}_{\tilde{N_1},\tilde{N_1}} & {\bf 0}_{\tilde{N_1},\tilde{N_1}} \hdots &   {\bf 0}_{\tilde{N_1},\tilde{N_1}} &  {\bf 0}_{\tilde{N_1},\tilde{N}} & {\bf 0}_{\tilde{N_1},\tilde{N}} & {\bf 0}_{\tilde{N_1},\tilde{N}}\\
	{\bf 0}_{\tilde{N_1},\tilde{N_1}} &
 e^{\Delta t\small{\mbox{diag}(\vec{v_{1}})\!\otimes \! A}}&
 {\bf 0}_{\tilde{N_1},\tilde{N_1}}\hdots &   {\bf 0}_{\tilde{N_1},\tilde{N_1}} & \vdots &  \vdots & \vdots\\
	\vdots & \ddots & \ddots  &  \vdots  &  \vdots & \vdots & \vdots\\
	{\bf 0}_{\tilde{N_1},\tilde{N_1}} & \hdots  &   {\bf 0}_{\tilde{N_1},\tilde{N_1}} & 
 e^{\Delta t\small{\mbox{diag}(\vec{v_{1}})\!\otimes \! A}} &
 {\bf 0}_{\tilde{N_1},\tilde{N}}  &  {\bf 0}_{\tilde{N_1},\tilde{N}}  & {\bf 0}_{\tilde{N_1},\tilde{N}} \\
{}_e{\cal B}_1 & {}_e{\cal B}_2 & \hdots & {}_e{\cal B}_{N_{v_2}} &   \frac{e^{A\Delta t}+e^{-A\Delta t}}{2}&
\frac{e^{A\Delta t}-e^{-A\Delta t}}{2}& {\bf 0}_{\tilde{N},\tilde{N}}    \\
{}_e{\cal E}_{2,1} & {}_e{\cal E}_{2,2} & \hdots & {}_e{\cal E}_{2,N_{v_2}} &    \frac{e^{A\Delta t}-e^{-A\Delta t}}{2}&
 \frac{e^{A\Delta t}+e^{-A\Delta t}}{2}& {\bf 0}_{\tilde{N},\tilde{N}}   \\
{}_e{\cal E}_{1} & {}_e{\cal E}_{1} & \hdots & {}_e{\cal E}_{1} & {\bf 0}_{\tilde{N},\tilde{N}} & {\bf 0}_{\tilde{N},\tilde{N}}  & {\bf 1} 
\end{array} 
\right)
$$
where ${}_e{\cal B}_{j_2}, {}_e{\cal E}_{2,j_2}, {}_e{\cal E}_{1} \in \mathbb{M}_{\tilde{N}, \tilde{N}_1}(\mathbb{R})$  
for $j_2=1, \dots, N_{v_2}$ are given by 
\begin{eqnarray*}
{}_e{\cal B}_{j_2} &=& \Delta v_1 \Delta v_2  \; v_{2,j_2}(A+\Pi)^{-1}\vec{\alpha},\nonumber\\
{}_e{\cal E}_{2,j_2} &=& \Delta v_1 \Delta v_2  \; v_{2,j_2}(A+\Pi)^{-1}\vec{\beta},\nonumber\\
{}_e{\cal E}_1 &=& \Delta v_1 \Delta v_2 (A+\Pi)^{-1} (\mathbbm{1}\otimes {\bf 1} -  e^{\Delta t\small{\mbox{diag}(\vec{v_{1}})\!\otimes \! A}})
\end{eqnarray*}
where the matrices $\vec{\alpha}, \vec{\beta}$ are given by 
\begin{eqnarray*}
 \vec{\alpha}&=&[\alpha_{1},\alpha_{2},\dots ,\alpha_{N_{v_1}}]\in \mathbb{M}_{\tilde{N},\tilde{N_1}}(\mathbb{R}) \mbox{ with } \alpha_{j_1}=\Big[ \frac{-e^{A\Delta t} }{2(1-v_{1, j_1})} - \frac{e^{-A\Delta t} }{2(1+v_{1, j_1})} + \frac{ e^{A \Delta t  v_{1, j_1}} }{(1-v_{1, j_1}^2)} \Big],\nonumber\\
\vec{\beta}&=&[\beta_{1},\beta_{2},\dots ,\beta_{N_{v_1}}]\in \mathbb{M}_{\tilde{N},\tilde{N_1}}(\mathbb{R}) \mbox{ with } \beta_{j_1}=\Big[  \frac{-e^{A\Delta t} }{2(1-v_{1, j_1})} + \frac{e^{-A\Delta t} }{2(1+v_{1, j_1})} + \frac{v_{1, j_1} e^{A \Delta t  v_{1, j_1}} }{(1-v_{1, j_1}^2)} \Big].
\end{eqnarray*}
\end{pro}
\begin{proof}
First the $(\tilde{N_1} N_{v_2})\times(\tilde{N_1} N_{v_2})$ block is diagonal and the diagonal part 
is $e^{\Delta t \small{\mbox{diag}(\vec{v_{1}})\!\otimes \! A}}$ ($N_{v_2}$ times). Second, the $3\times 3$ right bottom block corresponds to the homogeneous Maxwell equations. Its exponential can be computed and is equal to 
$$
\exp \left( 
\begin{array}{lll}
		{\bf 0}_{\tilde{N},\tilde{N}}& A\Delta t & {\bf 0}_{\tilde{N},\tilde{N}}\\
	A\Delta t & {\bf 0}_{\tilde{N},\tilde{N}} & {\bf 0}_{\tilde{N},\tilde{N}}\\
	{\bf 0}_{\tilde{N},\tilde{N}}& {\bf 0}_{\tilde{N},\tilde{N}} & {\bf 0}_{\tilde{N},\tilde{N}} 
\end{array}
\right) 
=
\left(
\begin{array}{lll}
	\frac{\exp(A\Delta t)+\exp(-A\Delta t)}{2} &\frac{\exp(A\Delta t)-\exp(-A\Delta t)}{2} &  {\bf 0}_{\tilde{N},\tilde{N}}\\
	\frac{\exp(A\Delta t)-\exp(-A\Delta t)}{2} & \frac{\exp(A\Delta t)+\exp(-A\Delta t)}{2}&  {\bf 0}_{\tilde{N},\tilde{N}}\\
	 {\bf 0}_{\tilde{N},\tilde{N}}&  {\bf 0}_{\tilde{N},\tilde{N}} & {\bf 1}
\end{array}
\right).
$$
Finally, we compute the three last block lines of $\exp(L\Delta t)$. 

\paragraph{Computation of ${}_e{\cal E}_{1}$
: solve $\mathbf{E_1}$\\}
First, we have for $\mathbf{f}_{j_1,j_2}(t)$  
$$
\mathbf{f}_{j_1,j_2}(t) = e^{v_{1, j_1} A(t-t^n) }\mathbf{f}_{j_1,j_2}(t^{n}), 
$$
which enables to compute $\mathbf{E_1}(t^{n+1})$
\begin{eqnarray*}
	\mathbf{E_1}(t^{n+1})&=&\mathbf{E_1}(t^{n}) -  \sum_{j_1, j_2} \int_{t^n}^{t^{n+1}} ({\bf 1}-\Pi) e^{v_{1, j_1}A (t-t^n)} dtv_{1, j_1}\mathbf{f}_{j_1,j_2}(t^{n}) \Delta v_1 \Delta  v_2\\
	&=& \mathbf{E_1}(t^{n}) + \Delta v_1 \Delta v_2(A+\Pi)^{-1}\sum_{j_1, j_2}({\bf 1}-  e^{ v_{1, j_1} A\Delta t} )\mathbf{f}_{j_1,j_2}(t^{n}), 
\end{eqnarray*}
from which we can thus deduce the last line of the exponential of the matrix.  

\paragraph{Computation of ${}_e{\cal B}_{j_2}, {}_e{\cal E}_{2, j_2}$: solve $\mathbf{B}, \mathbf{E_2}$\\}
Next, we focus on the calculation of $\mathbf{E_2}(t^{n+1})$ and $\mathbf{B}(t^{n+1})$ from known initial conditions 
$\mathbf{E_2}(t^{n})$ and $\mathbf{B}(t^{n})$. Let write down the equations  
for $(\mathbf{E_2}, \mathbf{B})(t)$ with  $t\in [t^n, t^{n+1}]$ 
\begin{eqnarray*}
	\frac{d}{dt} \mathbf{E_2}(t)&=& A \mathbf{B}(t) - \sum_{j_1, j_2} ({\bf 1}-\Pi) e^{v_{1, j_1} A(t-t^n)} v_{2, j_1}\mathbf{f}_{j_1,j_2}(t^{n}) \Delta v_1 \Delta  v_2 \\
	\frac{d}{dt} \mathbf{B}(t)&=&  A  \mathbf{E_2}(t) 
\end{eqnarray*}
which can be rewritten as $\frac{dU}{dt} =MU + R$ with $U(t)=(\mathbf{E_2}(t), \mathbf{B}(t))$ and 
\begin{eqnarray}
M&=&
\left( 
\begin{array}{ll}
		{\bf 0}_{\tilde{N},\tilde{N}}& A \\ 
	A & {\bf 0}_{\tilde{N},\tilde{N}}
\end{array}
\right), \;\; 
R(t)=\left( 
\begin{array}{ll} 
R_1(t)\\
{\bf 0}_{\tilde{N},1}
\end{array}
\right),   \nonumber\\
\label{def_M}
 R_1(t) &=& - \sum_{j_1, j_2}({\bf 1}-\Pi) e^{v_{1, j_1}A (t-t^n)} v_{2, j_1}\mathbf{f}_{j_1,j_2}(t^{n}) \Delta v_1 \Delta  v_2 
\end{eqnarray}
Thus, one can write the variation of constant formula 
\begin{equation}
	\label{duhamel_dg}
	U(t^{n+1}) = e^{M\Delta t} U(t^n) + \int_{t^n}^{t^{n+1}} e^{-M(t-t^{n+1})} R(t) dt.  
\end{equation}
First, $e^{M \Delta t}$ reads as, using its definition \eqref{def_M} 
$$
e^{M\Delta t} 
=
\frac{1}{2}\left(
\begin{array}{ll}
{e^{A\Delta t}+e^{-A\Delta t}} &{e^{A\Delta t}-e^{-A\Delta t}} \\
{e^{A\Delta t}-e^{-A\Delta t}}& {e^{A\Delta t}+e^{-A\Delta t}}
\end{array}
\right).  
$$
Second, one has to compute the integral  term in \eqref{duhamel_dg}
\begin{eqnarray*}
	\int_{t^n}^{t^{n+1}} e^{-M(t-t^{n+1})} R(t)dt &=& 
	\frac{1}{2} \left(
	\begin{array}{ll}
	 \int_{t^n}^{t^{n+1}} \Big[e^{-A(t-t^{n+1})}+e^{A(t-t^{n+1})}\Big]  R_1(t)dt \\
	  \int_{t^n}^{t^{n+1}} \Big[e^{-A(t-t^{n+1})}-e^{A(t-t^{n+1})}\Big] R_1(t) dt\\
	\end{array}
	\right)\nonumber\\
	&=& - \frac{1}{2} \left(
	\begin{array}{ll}
		\int_{t^n}^{t^{n+1}} \sum_{j_1, j_2} \Big[ e^{-A(t-t^{n+1})}+e^{A(t-t^{n+1})}\Big] e^{v_{1, j_1} A(t-t^n)} v_{2, j_1}{f}^{n}_{k,j_1,j_2} \Delta v_1 \Delta  v_2 dt  \\
		\int_{t^n}^{t^{n+1}} \sum_{j_1, j_2} \Big[ e^{-A(t-t^{n+1})}-e^{A(t-t^{n+1})} \Big] e^{v_{1, j_1} A(t-t^n)} v_{2, j_1}{f}^{n}_{k,j_1,j_2} \Delta v_1 \Delta  v_2 dt
	\end{array}
	\right)\nonumber\\
 &=& - \left(
	\begin{array}{ll} 
 \sum_{j_1, j_2}  {\cal I}_1 \; v_{2, j_1}{f}^{n}_{k,j_1,j_2} \Delta v_1 \Delta  v_2   \\
 \sum_{j_1, j_2}  {\cal I}_2 \; v_{2, j_1}{f}^{n}_{k,j_1,j_2} \Delta v_1 \Delta  v_2 
 \end{array}
	\right)
\end{eqnarray*}
where ${\cal I}_1, {\cal I}_2$ are given by 
\begin{eqnarray*}
{\cal I}_1	&=& \frac{1}{2}\int_{t^n}^{t^{n+1}} \Big[ e^{-A(t-t^{n+1})}+e^{A(t-t^{n+1})}  \Big] e^{v_{1, j_1} A(t-t^n)} dt  \nonumber\\
&=& \frac{(A+\Pi)^{-1}e^{A\Delta t} }{2(1-v_{1, j_1})} - \frac{(A+\Pi)^{-1}e^{-A\Delta t} }{2(1+v_{1, j_1})} - \frac{(A+\Pi)^{-1} v_{1, j_1} e^{A \Delta t  v_{1, j_1}} }{(1-v_{1, j_1}^2)}, \nonumber\\
{\cal I}_2	&=&\frac{1}{2} \int_{t^n}^{t^{n+1}} \Big[ e^{-A(t-t^{n+1})}-e^{A(t-t^{n+1})} \Big]  e^{v_{1, j_1} A(t-t^n)} dt \nonumber\\
&=& \frac{(A+\Pi)^{-1} e^{A\Delta t} }{2(1-v_{1, j_1})} + \frac{(A+\Pi)^{-1} e^{-A\Delta t} }{2(1+v_{1, j_1})} - \frac{(A+\Pi)^{-1} e^{A \Delta t  v_{1, j_1}}   }{(1-v_{1, j_1}^2)}. 
\end{eqnarray*}
Inserting these calculations in \eqref{duhamel_dg} leads to the following expression for $\mathbf{E_2}(t^{n+1})$ and $\mathbf{B}(t^{n+1})$
\begin{eqnarray*}
	\mathbf{E_2}(t^{n+1}) &=& \frac{1}{2}(e^{A\Delta t}+e^{-A\Delta t} ) \mathbf{E_2}(t^n) + \frac{1}{2}( e^{A\Delta t} - e^{-A\Delta t} ) \mathbf{B}(t^n)+ \Delta v_1 \Delta v_2(A+\Pi)^{-1} \sum_{j_1, j_2} v_{2, j_2} \beta_{j_1} \mathbf{f}_{j_1,j_2}(t^n) \nonumber\\
 	\mathbf{B}(t^{n+1}) &=& \frac{1}{2}(e^{A\Delta t}-e^{-A\Delta t} ) \mathbf{E_2}(t^n) + \frac{1}{2}( e^{A\Delta t} + e^{-A\Delta t} ) \mathbf{B}(t^n)+ \Delta v_1 \Delta v_2(A+\Pi)^{-1} \sum_{j_1, j_2} v_{2, j_2} \alpha_{j_1} \mathbf{f}_{j_1,j_2}(t^n)   
\end{eqnarray*}
where $\vec{\beta}=[\beta_{1},\beta_{2},\dots ,\beta_{N_{v_1}}]\in \mathbb{M}_{\tilde{N}\times\tilde{N_1}}(\mathbb{R})$ 
and $\vec{\alpha}=[\alpha_{1},\alpha_{2},\dots ,\alpha_{N_{v_1}}]\in \mathbb{M}_{\tilde{N}\times\tilde{N_1}}(\mathbb{R})$ 
are given by 
\begin{eqnarray*}
\beta_{j_1}&=&   \frac{-e^{A\Delta t} }{2(1-v_{1, j_1})} + \frac{e^{-A\Delta t} }{2(1+v_{1, j_1})} + \frac{v_{1, j_1} e^{A \Delta t  v_{1, j_1}} }{(1-v_{1, j_1}^2)}  \nonumber\\
\alpha_{j_1}&=& \frac{-e^{A\Delta t} }{2(1-v_{1, j_1})} - \frac{e^{-A\Delta t} }{2(1+v_{1, j_1})} + \frac{ e^{A \Delta t  v_{1, j_1}} }{(1-v_{1, j_1}^2)}, 
\end{eqnarray*}
which conclude the proof. 

\end{proof}

Shared the same spirit with Vlasov-Amp\`ere equation, we have the following discretized Poisson equation and error estimate for Vlasov-Maxwell equation with exponential Lawson RK DG FD discretization.

\begin{pro}
The exponential DG method \eqref{expo_DG_1} where the exponential is given in Prop \ref{prop_expo_2dv} (and its generalization to high order Lawson Runge-Kutta) satisfied by $U^n=(\vec{\bf{f}}, {\bf B}, {\bf E}_2, {\bf E}_1)^n$ preserves the following Poisson equation
$$
(A+\Pi)\mathbf{E}_1^n=-\sum_{j_1, j_2}({\bf 1}-\Pi)\mathbf{f}^n_{j_1,j_2}\Delta v_1\Delta v_2, \;\;\; \forall n\in \mathbb{N}^\star, 
$$
provided that it is satisfied at the initial time $n=0$. Here, $A$ is the DG matrix given by \eqref{dg_matrix}-\eqref{expression_MD}, 
$\Pi$ the orthogonal projection onto Ker$(A)$ and ${\bf{1}}$ 
is the identity matrix of size $\tilde{N}=(k+1)N_x$. 
\end{pro}

\begin{proof}
We present the proof for first order Lawson case (forward Euler), the proof can be generalized to arbitrary explicit Runge-Kutta scheme. First, we assume 
the Poisson equation 
$(A+\Pi)\mathbf{E}_1^0=-\sum_{j_1, j_2}({\bf 1}-\Pi)\mathbf{f}^0_{j_1,j_2}\Delta v_1\Delta v_2$ holds at the initial time. 

Next, from the scheme \eqref{expo_DG_1} and Prop \ref{prop_expo_2dv}, we have 
$$
\mathbf{f}_{j_1,j_2}^{n+1} = e^{v_{1, j_1} A\Delta t}\mathbf{f}_{j_1,j_2}^n-\Delta te^{v_{1, j_1} A\Delta t}(\tilde{\mathbfcal{F}} {\cal D} \mathbf{f})_{j_1,j_2}^n. 
$$
Regarding the $\mathbf{E}_1$ component, we have 
\begin{equation}
	\begin{aligned}
		\mathbf{E}^{n+1}_1&=\mathbf{E}^{n}_1 + \Delta v_1 \Delta v_2(A+\Pi)^{-1}\sum_{j_1, j_2}({\bf 1}-  e^{ v_{1, j_1} A\Delta t} )\mathbf{f}_{j_1,j_2}^{n}\nonumber\\
  &-\Delta t\Delta v_1 \Delta v_2(A+\Pi)^{-1}\sum_{j_1, j_2}({\bf 1} -  e^{ v_{1, j_1} A\Delta t} )(\tilde{\mathbfcal{F}}{\cal D} \mathbf{f})_{j_1,j_2}^n\\
  &=\mathbf{E}^{n}_1 + \Delta v_1 \Delta v_2(A+\Pi)^{-1}\sum_{j_1, j_2}({\bf 1}-\Pi)({\bf 1}-  e^{ v_{1, j_1} A\Delta t} )\mathbf{f}_{j_1,j_2}^{n}\nonumber\\
  &-\Delta t\Delta v_1 \Delta v_2(A+\Pi)^{-1}\sum_{j_1, j_2}({\bf 1}-\Pi)({\bf 1} -  e^{ v_{1, j_1} A\Delta t} )(\tilde{\mathbfcal{F}}{\cal D} \mathbf{f})_{j_1,j_2}^n\\
	&=\mathbf{E}^{n}_1 + \Delta v_1 \Delta v_2(A+\Pi)^{-1}\sum_{j_1, j_2}({\bf 1}-\Pi)({\bf 1}-  e^{ v_{1, j_1} A\Delta t} )\mathbf{f}_{j_1,j_2}^{n}\\
 &-\Delta t\Delta v_1 \Delta v_2(A+\Pi)^{-1}({\bf 1}-\Pi)\sum_{j_1, j_2}(\tilde{\mathbfcal{F}}{\cal D} \mathbf{f})_{j_1,j_2}^n\\
 &+\Delta v_1 \Delta v_2(A+\Pi)^{-1}\sum_{j_1, j_2}({\bf 1}-\Pi)( e^{v_{1, j_1} A\Delta t}\mathbf{f}_{j_1,j_2}^n-\mathbf{f}_{j_1,j_2}^{n+1})\\
 \label{eq:Poissoneequationpreserving}
 &=\mathbf{E}^{n}_1 + \Delta v_1 \Delta v_2(A+\Pi)^{-1}\sum_{j_1, j_2}(({\bf 1}-\Pi)\mathbf{f}_{j_1,j_2}^{n}-({\bf 1}-\Pi)\mathbf{f}_{j_1,j_2}^{n+1}).
	\end{aligned}
\end{equation} 
By induction, if the discrete Poisson equation is satisfied at iteration $n$, then it is satisfied at iteration $n+1$ and the proof is complete.

\end{proof}

\section{Vlasov-Maxwell 2dx-2dv}
\label{sec:DG2Dx}
We finally consider the 2dx-2dv Vlasov-Maxwell model satisfied by \\ 
$f(t, x, y, v_1, v_2), E_1(t, x, y), E_2(t, x, y), B(t, x, y)$, with $t\geq 0, (x,y)\in [0,L_x]\times [0, L_y]$ 
and $(v_1, v_2)\in \mathbb{R}^2$
\begin{equation}
\label{eq:model2dx}
	\left\{
	\begin{aligned}
		&\partial_t f+v_1\partial_x f + v_2\partial_y f+ E_1 \partial_{v_1} f +E_2 \partial_{v_2} f +B(v_2 \partial_{v_1} f -v_1 \partial_{v_2} f )=0, \\
		&\partial_t B =\partial_y E_1 - \partial_x E_2, \\
		&\partial_t E_1= \partial_y B - \int_ {\mathbb{R}^2} v_1f \mathrm{d}v_1\mathrm{d}v_2+\bar{J}_1,\\
		&\partial_t E_2=-\partial_x B-\int_ {\mathbb{R}^2} v_2f \mathrm{d}v_1\mathrm{d}v_2+\bar{J}_2,\\ 
  & \partial_x E_1+\partial_y E_2 = \int_ {\mathbb{R}^2}f \mathrm{d}v_1\mathrm{d}v_2 -\bar{\rho}, \;\;\; \partial_x B+\partial_y B=0,  
	\end{aligned}
\right.
\end{equation} 
with initial conditions $(f_0(x,y,v_1,v_2), E^0_1(x,y), E_2^0(x,y), B(x,y))$ such that the Poisson equation is satisfied $\partial_x E_1^0+\partial_y E_2^0 = \int_ {\mathbb{R}^2}f_0 \mathrm{d}v_1\mathrm{d}v_2 -\bar{\rho}$ and periodic boundary conditions are considered in space. here, $\bar{J}_i=\frac{1}{L_x L_y}\int_{L_x\times L_y} \int_{\mathbb{R}^2} v_i f dx dy dv_1 dv_2$, $\bar{\rho} = \frac{1}{L_x L_y}\int_{L_x\times L_y} \int_{\mathbb{R}^2} f dx dy dv_1 dv_2$, $E=(E_1, E_2), B=(B_1, B_2), \nabla=(\partial_x, \partial_y)$. 

\subsection{Exponential DG discretization}
Here we apply 2D DG method in $(x,y)$ direction and consider the discretization on Cartesian meshes with a rectangular triangulation $I_i \times I_j$. We define the space $V_h$ as the space of tensor product piece-wise polynomials of degree at most $k$ in each variable on every element, i.e. $V_h^k = \{ v_h:  v_h|_{I_i\times I_j} \in Q^k(I_i\times I_j) \}$, where $Q^k(I_i\times I_j)$ is the space of tensor products of one dimensional polynomials of degree up to $k$.


We follow the lines of the above subsections: we use a DG method in the 2D space direction $(x,y)$ (with $N_x$ (resp.  $N_y$) cells in direction $x$ (resp. $y$) and a grid in the velocity direction $v_{\ell,j_\ell} = v_{\ell, \min} + j_\ell \Delta v_\ell, \ell=1, 2, \;\; j_\ell=1, \dots, N_{v}$. 
The 2D DG approximation for $f$ is represented as 
(for $j_1, j_2=1, \dots, N_v$) 
$$
f(t,x,y, v_{j_1}, v_{j_2})\approx f_h(t,x,y, v_{j_1}, v_{j_2})= \sum_{i=1}^{N_x}\sum_{j=1}^{N_y} \sum_{m=0}^{k}\sum_{n =0}^{k}{f}_{ij}^{mn}(t,v_{j_1},v_{j_2}) \xi_i^m(x)\xi_j^n(y).
$$ 
For simplicity, we only present the 2D DG discretization for linear part of equation of $f$ obtained by multiplying the 
Vlasov equation by $\xi^\ell(x)\xi^s(y)$ (for $\ell=0, \dots, k$ and $s=0, \dots, k$) on $(x,y)\in I_i\times I_j$ 
(for $i=1,2,...N_x$ and $j=1,2,...N_y$):
\begin{equation}
\begin{aligned}
&\sum_{m=0}^{k}\sum_{n=0}^{k}\left[\partial_t{f}_{ij}^{mn}(t,v_{j_1},v_{j_2}) (\xi^m, \xi^\ell)_{I_i}(\xi^n, \xi^s)_{I_j}\right]\\
  & - \sum_{m=0}^{k}\sum_{n=0}^{k}\left[v_{j_1}{f}_{ij}^{mn}(t,v_{j_1},v_{j_2})(\xi^m,\partial_x\xi^\ell)_{I_i}(\xi^n, \xi^s)_{I_j}+v_{j_2}{f}_{ij}^{mn}(t,v_{j_1},v_{j_2})(\xi^m,\xi^\ell)_{I_i}(\xi^n, \partial_y\xi^s)_{I_j}\right] \\
  &+\sum_{m=0}^{k}\sum_{n=0}^{k}\left[(v_{j_1}\Big[\{f_h(t, x, y, v_{j_1}, v_{j_2}\}\xi^\ell\Big]^{i+\frac{1}{2}}_{i-\frac{1}{2}},\xi^s)_{I_j} +(v_{j_2}\Big[\{f_h(t, x, y, v_{j_1}, v_{j_2}\}\xi^s\Big]^{j+\frac{1}{2}}_{j-\frac{1}{2}},\xi^\ell)_{I_i}\right]=0, 
\end{aligned}
\end{equation}
where we used the central fluxes in $x$ and $y$, that is for the $x$ direction  \\
$\{f_h(t, x, y,v_{j_1}, v_{j_2})\}|_{x_{i\pm 1/2}} = \frac{1}{2}(f_h(t,x_{i\pm 1/2}^+ ,y,v_{j_1}, v_{j_2})+f_h(t,x_{i\pm 1/2}^-,y,v_{j_1}, v_{j_2}))$. 

We consider $\mathbf{f}_{j_1,j_2}\in \mathbb{R}^{(k+1)^2N_xN_y}$ the vector containing the degree of freedom $f^{m,n}_{i,j}$ 
\begin{equation}
\label{fDG2D}
\mathbf{f}_{j_1,j_2} = [f^{0,0}_{1,1},f^{1,0}_{1,1}, \dots, f^{k,0}_{1,1}, \dots, f^{k,0}_{N_x,1}, f^{0,1}_{1,1},\dots, f^{k,k}_{N_x,N_y}]_{j_1,j_2}^T
\end{equation}
and $\mathbf{E_1}, \mathbf{E_2}, \mathbf{B}\in \mathbb{R}^{(k+1)^2 N_x N_y}$ the vectors (defined as \eqref{fDG2D}) containing the DG degree of freedom of  the electromagnetic fields 
$(E_1, E_2, B)$. Finally, we introduce $(\tilde{\mathbfcal F} {\cal D} \mathbf{f})_{j_1,j_2}$ the DG approximation of the nonlinear term $({\cal F} \cdot \nabla_v f)(v_{j_1},v_{j_2}), \ \mbox{ with }  {\cal F} =({E_1}+ {B} v_2, {E_2}  -{B} v_1)$ using similar techniques as in the 1dx case. With these notations, we have the following semi-discretized scheme 
\begin{equation}
	\label{eq:model_DG2Dx}
	\left\{
	\begin{aligned}
		&\partial_t \mathbf{f}_{j_1,j_2}=({\bf 1}_y \otimes v_{1,j_1}A^x)\mathbf{f}_{j_1,j_2}+(v_{2,j_2}A^y \otimes {\bf 1}_x)\mathbf{f}_{j_1,j_2}- 
  (\tilde{\mathbfcal F} {\cal D} \mathbf{f})_{j_1,j_2}\\ 
		&\partial_t \mathbf{B} =-(A^y \otimes {\bf 1}_x) \mathbf{E_1} + ({\bf 1}_y \otimes A^x)  \mathbf{E_2}, \\
		&\partial_t \mathbf{E_1}=-(A^y\otimes {\bf 1}_x) \mathbf{B}-\sum_{j_1, j_2} v_{1, j_1} {\bf P}\mathbf{f}_{j_1,j_2} \Delta v_1\Delta  v_2,\\
		&\partial_t \mathbf{E_2}=({\bf 1}_y\otimes A^x)\mathbf{B}- \sum_{j_1, j_2} v_{2, j_2} {\bf P}\mathbf{f}_{j_1,j_2} \Delta v_1\Delta v_2,
	\end{aligned}
	\right.
\end{equation} 
where $A^x\in\mathbb{M}_{(k+1)N_x, (k+1)N_x}(\mathbb{R})$ and $A^y\in\mathbb{M}_{(k+1)N_y, (k+1)N_y}(\mathbb{R})$ are the matrices coming from the DG semi-discretization in each space direction as before, $\otimes$ is the Kronecker product, 
${\bf P}=({\bf 1}_y-\Pi_y) \otimes ({\bf 1}_x - \Pi_x)$ 
with ${\bf 1}_{x}$ (resp. ${\bf 1}_{y})$ the identity matrix of size $(k+1)N_{x}$ (resp. $(k+1)N_{y}$) and $\Pi_{x}$ (resp. $\Pi_{y}$) the projection onto Ker$(A^{x})$ (resp. Ker$(A^{y})$).

Before discussing the time discretization, we prove the following proposition.  
\begin{pro}
\label{pro_DG2D_poisson}
The semi-discretized system \eqref{eq:model_DG2Dx} satisfied by $({{\bf f}}_{j_1,j_2}, {\bf B}, {\bf E}_1, {\bf E}_2)(t)$ preserves the following discretized Poisson equation  
$$
({\bf 1}_y \otimes A^x) \mathbf{E_1}(t)+(A^y \otimes {\bf 1}_x) \mathbf{E_2}(t)=-\sum_{j_1, j_2}{\bf P}\mathbf{f}_{j_1,j_2}(t) \Delta v_1\Delta v_2, \;\;\; \forall t>0, 
$$ 
provided that it is satisfied at time $t=0$.
\end{pro}
\begin{proof}
Let derive with respect to time the left hand side of the equality to get 
    \begin{eqnarray*}
\partial_t(({\bf 1}_y \otimes A^x) \mathbf{E_1}+(A^y \otimes {\bf 1}_x) \mathbf{E_2})&=& ({\bf 1}_y \otimes A^x) \partial_t\mathbf{E_1}+(A^y \otimes {\bf 1}_x) \partial_t\mathbf{E_2}\\
        &=&({\bf 1}_y \otimes A^x) (-(A^y \otimes {\bf 1}_x) \mathbf{B}-\sum_{j_1, j_2} v_{1, j_1} {\bf P}\mathbf{f}_{j_1,j_2} \Delta v_1\Delta  v_2)\\
        &&+(A^y \otimes {\bf 1}_x) (({\bf 1}_y \otimes A^x)\mathbf{B}- \sum_{j_1, j_2} v_{2, j_2} {\bf P}\mathbf{f}_{j_1,j_2} \Delta v_1\Delta v_2)\\
        &=& \Big[(A^y \otimes {\bf 1}_x) ({\bf 1}_y \otimes A^x) - ({\bf 1}_y \otimes A^x)(A^y \otimes {\bf 1}_x)\Big]\mathbf{B} \\
        && -({\bf 1}_y \otimes A^x) \sum_{j_1, j_2} v_{1, j_1} {\bf P}\mathbf{f}_{j_1,j_2} \Delta v_1\Delta  v_2) \\
        &&- (A^y \otimes {\bf 1}_x)\sum_{j_1, j_2} v_{2, j_2} {\bf P}\mathbf{f}_{j_1,j_2} \Delta v_1\Delta v_2)\\
        &=&-\sum_{j_1, j_2}   \Big(\partial_t ({\bf P}\mathbf{f}_{j_1,j_2})+{\bf P}( \tilde{\mathbfcal F}{\cal D} \mathbf{f})_{j_1,j_2}\Big)\Delta v_1\Delta  v_2\\
        &=&-\partial_t \Big(\sum_{j_1, j_2} {\bf P}\mathbf{f}_{j_1,j_2} \Delta v_1\Delta  v_2\Big), 
\end{eqnarray*} 
where we used the identities 
\begin{eqnarray*}
(A^y\otimes {\bf 1}_x)({\bf 1}_y\otimes A^x) &=& (A^y{\bf 1}_y)\otimes ({\bf 1}_x A^x) = A^y\otimes A^x\nonumber\\ 
({\bf 1}_y \otimes A^x)(A^y \otimes {\bf 1}_x) &=& ({\bf 1}_y A^y)\otimes (A^x {\bf 1}_x) = A^y\otimes A^x
\end{eqnarray*} 
to pass from the third to the fourth equality. 
Integrating in time the obtained equality and assuming 
the discrete Poisson equation holds at time $t=0$ leads to 
the result. 
\end{proof}
We end this part by giving some elements on the time discretization. First, in this case, it is difficult to compute the exponential of the linear part. However, we can consider 
the exponential of the ${\bf f}_{j_1, j_2}$ linear part 
(which corresponds to the $(x,y)$ transport). Indeed, 
we observe from \eqref{eq:model_DG2Dx} that this linear part writes 
$$
\partial_t {\bf f}_{j_1, j_2} = \Big[({\bf 1}_y \otimes v_{1,j_1}A^x)+(v_{2,j_2}A^y \otimes {\bf 1}_x)\Big] \mathbf{f}_{j_1,j_2} = \Big[v_{2,j_2}A^y\oplus v_{1,j_1}A^x \Big]{\bf f}_{j_1, j_2},  
$$
where we used the definition of the Kronecker sum $\oplus$. The exact solution can be then written as 
$$
{\bf f}_{j_1, j_2}(t) = \exp\Big( (v_{2,j_2}A^y\oplus v_{1,j_1}A^x) \, t\Big) {\bf f}_{j_1, j_2}(0) = \exp\Big( v_{2,j_2} t A^y\oplus v_{1,j_1} t A^x \Big) {\bf f}_{j_1, j_2}(0).  
$$
It is well known that the exponential of a matrix with Kronecker sum structure is equal to the Kronecker product
of the exponentials that is 
$$
{\bf f}_{j_1, j_2}(t) = \exp( v_{2,j_2}t A^y)\otimes \exp( v_{1,j_1}t A^x) {\bf f}_{j_1, j_2}(0), 
$$
which can be recast using the vec operation as 
\begin{equation}
\label{vec_kron}
\Large{\textgoth{f}}_{j_1,j_2}(t)= \exp(v_{2,j_2} t A^y) \Large{\textgoth{f}}_{j_1,j_2}(0) \exp(\Delta v_{1,j_1}t (A^x)^T),  
\end{equation}
where ${\bf f}_{j_1, j_2}(s)$=vec$(\Large{\textgoth{f}}_{j_1,j_2}(s))$ 
(for $s=0, t$ denotes the vectorization operation which takes the  matrix 
$\Large{\textgoth{f}}_{j_1,j_2}\in\mathbb{M}_{(k+1)N_x, (k+1)N_y}(\mathbb{R})$ as entry and gives the vector ${\bf f}_{j_1, j_2}\in \mathbb{R}^{(k+1)^2N_xN_y}$ as a result. 
This means that the update of $f$ requires matrix-vector 
products operations that only involves 
to assembly exponential of matrices $A^x$ and $A^y$ which are computed from the one-dimensional case (see \eqref{dg_matrix}). 
Moreover, these matrix-vector products calculations can be 
performed in a very efficient way. 
This nice property has been exploited in the literature to design 
efficient routines for computing matrix exponentials \cite{caliari2022mu,croci2023exploiting,munoz2022exploiting}. 

\begin{remark}
The semi-discretized Vlasov-Maxwell system \eqref{eq:model_DG2Dx} 
can be degenerated to 
a semi-discretization of the Vlasov-Poisson system satisfied by 
$(f, E_1, E_2)$. In this case, the Lawson scheme only applies to 
the unknown ${\bf f}$ and then requires the calculation of 
$\exp(v_{2,j_2} t A^y \oplus v_{1,j_1} t A^x)$ which can be performed efficiently thanks to  \eqref{vec_kron}. The update of the electric field is performed using the Poisson equation thanks to the updated 
${\bf f}$. 
\end{remark}

\subsection{Fourier based space discretization}

In this part, we consider Fourier in space combined with finite 
differences in velocity to semi-discretize the Vlasov-Maxwell 
system \eqref{eq:model2dx} and we will see that in this case, 
it will be possible to compute explicitely the exponential 
of the linear part. 

Denoting $\hat{f}_{k_x, k_y, j_1,j_2}$ the Fourier coefficient of $f$ in space and evaluated at the velocity grid introduced previously ($k_x,k_y$ being the Fourier variables), $\hat{E}_{1,k_x,k_y}, \hat{E}_{2,k_x,k_y}, \hat{B}_{k_x,k_y}$ the Fourier coefficients of  the electromagnetic fields 
$({E}_1, {E}_2, {B})$, and introducing 
the force term ${\cal F} =(E_1 + B v_2, E_2  -B v_1)$, 
we get the following semi-discretized scheme 
\begin{equation}
\label{eq:model_fourier2dx}
	\left\{
	\begin{aligned}
		&\partial_t \hat{f}_{k_x,k_y,j_1,j_2}+(v_{1,j_1}i k_x+v_{2,j_2}i k_y) \hat{f}_{k_x,k_y, j_1, j_2}+ \widehat{({\cal F}  {\cal D} f)}_{k_x,k_y,j_1,j_2} = 0, \\
		&\partial_t \hat{B}_{k_x,k_y} =ik_y \hat{E}_{1, k_x,k_y}- ik_x \hat{E}_{2,k_x,k_y}, \\
		&\partial_t \hat{E}_{1,k_x,k_y}=ik_y \hat{B}_{k_x,k_y}-\sum_{j_1, j_2} v_{1, j_1} \hat{f}_{k_x,k_y,j_1,j_2} \Delta v_1\Delta  v_2+\bar{J}_1,\\
		&\partial_t \hat{E}_{2,k_x,k_y}=-ik_x \hat{B}_{k_x,k_y}- \sum_{j_1, j_2} v_{2, j_2} \hat{f}_{k_x,k_y,j_1,j_2} \Delta v_1\Delta v_2+\bar{J}_2,
	\end{aligned}
\right.
\end{equation} 
with the initial conditions $\hat{f}_{k_x,k_y,j_1,j_2}(0), \hat{B}_{k_x,k_y}(0), \hat{E}_{1,k_x,k_y}(0), \hat{E}_{2,k_x,k_y}(0)$ 
satisfying the Poisson equation $ik_x\hat{E}_{1,k_x,k_y}(0) + ik_y\hat{E}_{2,k_x,k_y}(0)=\sum_{j_1,j_2}\hat{f}_{k_x,k_y,j_1,j_2}(0)\Delta v_1\Delta v_2$ for $(k_x,k_y)\neq (0,0)$. 

For the semi-discretized system \eqref{eq:model_fourier2dx}, we have a similar 
proposition as Prop \eqref{pro_DG2D_poisson} in this Fourier case.  
\begin{pro}
The semi-discretized system \eqref{eq:model_fourier2dx} satisfied by $(\hat{f}_{k_x,k_y,j_1,j_2}, \hat{B}_{k_x,k_y}, \hat{E}_{1,k_x,k_y}, \hat{E}_{2,k_x,k_y})(t)$ preserves the following discretized Poisson equation 
$$
ik_x\hat{E}_{1,k_x,k_y}(t)+ik_y\hat{E}_{2,k_x,k_y}(t)=\sum_{j_1, j_2} \hat{f}_{k_x,k_y,j_1,j_2}(t) \Delta v_1\Delta v_2, \;\;\; (k_x, k_y)\neq (0,0), \forall t>0,  
$$  
provided it is satisfied at time $t=0$. 
\end{pro}
\begin{proof}
As in the proof of Prop \eqref{pro_DG2D_poisson}, we take the derivative with respect to time of the left hand side to get 
\begin{eqnarray*}
\partial_t(ik_x\hat{E}_{1,k_x,k_y}+ik_y\hat{E}_{2,k_x,k_y})&=& ik_x\Big(ik_y \hat{B}_{k_x,k_y}-\sum_{j_1, j_2} v_{1, j_1} \hat{f}_{k_x,k_y,j_1,j_2} \Delta v_1\Delta  v_2\Big)\nonumber\\
        && + ik_y\Big(-ik_x \hat{B}_{k_x,k_y}- \sum_{j_1, j_2} v_{2, j_2} \hat{f}_{k_x,k_y,j_1,j_2} \Delta v_1\Delta v_2\Big)\nonumber\\ 
        &=&\sum_{j_1, j_2}  \Delta v_1\Delta  v_2 \Big(\partial_t \hat{f}_{k_x,k_y,j_1,j_2}+\widehat{({\cal F} {\cal D} f)}_{k_x,k_y,j_1,j_2}\Big)\nonumber\\
        &=&\partial_t \Big(\sum_{j_1, j_2} \hat{f}_{k_x,k_y,j_1,j_2} \Delta v_1\Delta  v_2 \Big), 
\end{eqnarray*} 
where we used the summation on $j_1,j_2$ of the Vlasov equation 
together with the fact that ${\cal D}$ is a conservative finite difference operator. Finally, integrating the result in time between $0$ and $t$ and assuming the relation holds at $t=0$ 
ends the proof. 
\end{proof}

To derive a fully discrete scheme, we introduce as previously 
the vector ${\hat{\bf f}}_{k_x,k_y}\in\mathbb{C}^{N_{v_1}N_{v_2}}$  and denote $U(t)=({\hat{\bf f}}, {\hat{ B}}, {\hat{ E}_2}, {\hat{ E}_1})_{k_x,k_y}(t) \in \mathbb{M}_{N_{v_1}N_{v_2}+3,N_{v_1}N_{v_2}+3}(\mathbb{C})$, then the system \eqref{eq:model_fourier2dx} can be rewritten as 
\begin{equation}
\label{ULU2d}
\partial_t U= L U + N(U), 
\end{equation}
with 
\begin{equation}
\label{def_L_2Dfourier}
\hspace{-0.15cm}L\hspace{-0.1cm}=\hspace{-0.1cm} \left( 
 \begin{array}{lllllllccc}
-i(\!k_x\mbox{diag}(\vec{v_{1}})\!\!+\!\! k_y v_{2, 1}\!) & {\bf 0}_{N_{v_1},N_{v_1}} & \hdots &   {\bf 0}_{N_{v_1},N_{v_1}} &  {\bf 0}_{N_{v_1},1} & \dots & {\bf 0}_{N_{v_1},1}\\
 {\bf 0}_{N_{v_1},N_{v_1}} & -i(\!k_x\mbox{diag}(\vec{v_{1}})\!\! +\!\! k_y v_{2, 2}\!) &  \hdots &   {\bf 0}_{N_{v_1},N_{v_1}} & \vdots &  \dots & \vdots\\
\vdots & \ddots & \ddots  &  {\bf 0}_{N_{v_1},N_{v_1}}  &  \vdots & \dots & \vdots\\
 {\bf 0}_{N_{v_1},N_{v_1}} & \hdots  &   {\bf 0}_{N_{v_1},N_{v_1}} & -i(\!k_x\mbox{diag}(\vec{v_{1}})\!\! +\!\! k_y v_{2, N_{v_2}}\!) &  {\bf 0}_{N_{v_1},1}  &  \dots  &  {\bf 0}_{N_{v_1},1} \\
 {\bf 0}_{1,N_{v_1}} & \hdots  & {\bf 0}_{1,N_{v_1}} & {\bf 0}_{1,N_{v_1}} & 0 &   -ik_x & ik_y  \\
 -\Delta v_1\Delta v_2 v_{2,1} {\bf 1} & -\Delta v_1\Delta v_2 v_{2,2} {\bf 1} & \hdots & -\Delta v_1\Delta v_2 v_{2,N_{v_y}} {\bf 1} & -ik_x & 0 & 0\\
 - \Delta v_1\Delta v_2 \vec{v_1} & -\Delta v_1\Delta v_2 \vec{v_1} & \hdots  & -\Delta v_1\Delta v_2 \vec{v_1} & ik_y& 0& 0 
\end{array} 
 \right) 
  \end{equation}  
  and 
\begin{eqnarray*}
\hat{\bf f}_{k_x,k_y} = \left( 
 \begin{array}{lllllll} 
\hat{f}_{k_x,k_y,1,1}\\
\vdots \\
\hat{f}_{k_x,k_y,N_{v_1},1}\\
\vdots \\
\hat{f}_{k_x,k_y,N_{v_1},N_{v_2}}
\end{array} 
 \right)\in \mathbb{M}_{N_{v_1}\!N_{v_2},N_{v_1}\!N_{v_2}}(\mathbb{C})
\;\; \mbox{ and } \;\; 
N(U)=&\hspace{-0.35cm} \left( 
 \begin{array}{lllllll}  
(\widehat{{\cal F {\cal D} {\bf f}}})_{k_x,k_y}\\
0\\
0\\
0
 \end{array} 
 \right) \in \mathbb{M}_{N_{v_1}\!N_{v_2}\!+3,N_{v_1}\!N_{v_2}\!+3}(\mathbb{C}), 
  \end{eqnarray*}  
 where we denote $\vec{v_1}$ the vector with components $(\vec{v_1})_{j_1}=v_{1, \min} + j_1 \Delta v_1$ and 
 ${v_{2,j_2}}=v_{2, \min} + j_2 \Delta v_2$. In the same spirit 
 as previously, 
diag$(\vec{v_1})$ denotes the diagonal matrix whose diagonal is composed of $\vec{v_1}$. 

We now study the time discretization of \eqref{ULU2d} based on a Lawson scheme which requires to know $\exp(tL)$ with $L$ given above. 
Similar (but more tedious) calculations to those performed in the proof 
of Prop \ref{prop_expo_2dv} enable to get an explicit expression 
of $\exp(tL)$. To end this section, we prove that the following scheme,  with $U^n\approx U(t^n), t^n=n\Delta t, \Delta t>0$ and the notations 
introduced above 
\begin{equation}
\label{lawson_fourier2D}
U^{n+1} = \exp(\Delta t L)U^n + \Delta t \exp(\Delta t L)N(U^n),  
\end{equation}
that approximates the ODE \eqref{ULU2d} preserves a discrete Poisson equation. This is the object of the following proposition. 

\begin{pro}
The Lawson scheme \eqref{lawson_fourier2D} satisfied by $U^n=(\hat{\bf f}, \hat{B}, \hat{E}_2, \hat{E}_1)_{k_x,k_y}^n$ preserves 
the following Poisson equation 
$$
ik_x \hat{E}_{1,k_x,k_y}^n+ik_y \hat{E}_{2,k_x,k_y}^n = \sum_{j_1,j_2} \hat{f}^n_{k_x,k_y,j_1,j_2}\Delta v_1\Delta v_2, \;\; (k_x,k_y)\neq (0,0), \forall n\in \mathbb{N}^\star, 
$$
provided that it is satisfied at the initial time $n=0$. 
\end{pro}

\begin{proof}
First, we need to know the shape of $\exp(\Delta tL)$. From the one-dimensional calculations and from \cite{boutin2022modified}, we have  
$$
 \!\!\!\!\!\!\!\!\!\!\!\!e^{L\Delta t} \!\!= \!\!\!
 \left( 
 \begin{array}{lllllllccc}
e^{-i  \Delta t(k_x \vec{v_{1}} + k_y v_{2, 1}) }&\!\!\! {\bf 0}_{N_{v_1},N_{v_1}} &\!\!\! {\bf 0}_{N_{v_1},N_{v_1}}\hdots &\!\!\!   {\bf 0}_{N_{v_1},N_{v_1}} &\!\!\!  {\bf 0}_{N_{v_1},1} &\!\!\! {\bf 0}_{N_{v_1},1} &\!\!\! {\bf 0}_{N_{v_1},1}\\
 {\bf 0}_{N_{v_1},N_{v_1}} &\!\!\!\!\!\!e^{-i \Delta t (k_x  \vec{v_{1}} + k_y v_{2,2}) } &\!\!\!  {\bf 0}_{N_{v_1},N_{v_1}}\hdots &\!\!\!   {\bf 0}_{N_{v_1},N_{v_1}} &\!\!\! \vdots &\!\!\!  \vdots &\!\!\! \vdots\\
\vdots & \!\!\!\!\!\!\ddots &\!\!\! \ddots  &\!\!\!  {\bf 0}_{N_{v_1},N_{v_1}}  &\!\!\!  \vdots &\!\!\! \vdots &\!\!\!  \vdots\\
 {\bf 0}_{N_{v_1},N_{v_1}} &\!\!\!\!\!\! \hdots  &\!\!\!\!\!\!   {\bf 0}_{N_{v_1},N_{v_1}} & \!\!\!\!\!\!e^{-i \Delta t (k_x  \vec{v_{1}} + k_y v_{2,N_{v_2}}) }  & \!\!\! {\bf 0}_{N_{v_1},1}  &  {\bf 0}_{N_{v_1},1}  &  {\bf 0}_{N_{v_1},1} \\
{}_e{\cal B}_{k_x,k_y,\star, 1} &\!\!\!\!\!\! {}_e{\cal B}_{k_x,k_y,\star, 2}  &\!\!\!\!\!\! \hdots &\!\!\!\!\!\! {}_e{\cal B}_{k_x,k_y,\star, N_{v_2}} & \!\!\!\!\!\! \cos(|k| \Delta t) &\!\!\!\!\!\! \frac{-i k_x\!\sin(|k| \Delta t)}{|k|}& \!\!\!\!\!\!  \frac{i k_y\!\sin(|k| \Delta t)}{|k|} \!\!\!\!\!\!\!\!\!\!\!\!  \\
{}_e{\cal E}_{2,k_x,k_y,\star, 1} &\!\!\!\!\!\! {}_e{\cal E}_{2,k_x,k_y,\star, 2}  &\!\!\!\!\!\! \hdots &\!\!\!\!\!\! {}_e{\cal E}_{2,k_x,k_y,\star, N_{v_2}} & \!\!\!\!\!\!\!  \frac{-i k_x\!\sin(|k| \Delta t)}{|k|}  & \!\!\!\!\! \frac{k_x^2\!\cos(|k|\Delta t) + k_y^2}{|k|^2} &\!\!\!\!\!\! \frac{k_xk_y\!(1\!-\!\cos(\Delta t|k|)}{|k|^2}   \\
{}_e{\cal E}_{1,k_x,k_y,\star, 1} &\!\!\!\!\!\! {}_e{\cal E}_{1,k_x,k_y,\star, 2}  &\!\!\!\!\!\! \hdots &\!\!\!\!\!\! {}_e{\cal E}_{1,k_x,k_y,\star, N_{v_2}}  &\!\!\!\!\!\!\!  \frac{i k_y\!\sin(|k| \Delta t)}{|k|} &\!\!\!\!\!\!\frac{ k_xk_y(1-\cos(\Delta t|k|) }{|k|^2} &\!\!\!\!\! \frac{k_y^2 \cos(\Delta t|k|) + k_x^2}{|k|^2}\!\!\!\!\!\!\!\!\!\!\!\!
\end{array} 
 \right)
$$
where we used the fact that the exponential of the homogeneous Maxwell part is 
$$
\exp\left(   t \left( 
 \begin{array}{lllllllccc}
0 & -ik_x & ik_y\\
-ik_x & 0 & 0 \\
ik_y & 0 &0
\end{array} 
 \right)
 \right)
 =
 \left( 
 \begin{array}{lllllllccc}
\cos( t |k|) &  -\frac{i k_x\sin( t |k|)}{|k|} & \frac{ i k_y \sin( t |k|)}{|k|}\\
-\frac{i k_x\sin( t |k|)}{|k|}  & \frac{k_x^2 \cos( t |k|)+k^2_y}{|k|^2}& \frac{k_xk_y (1- \cos( t|k|))}{|k|^2} \\
\frac{ i k_y \sin( t |k|)}{|k|} & \frac{k_xk_y (1- \cos( t|k|))}{|k|^2}   &  \frac{k_y^2 \cos( t |k|)+k_x^2}{|k|^2}
\end{array} 
 \right),  
$$
and the vectors $({}_e{\cal B}, {}_e{\cal E}_{2}, {}_e{\cal E}_{1})_{k_x,k_y,\star, j_2} \in \mathbb{\cal C}^{3 N_{v_1}}$ for all 
$k_x,k_y$ and $j_2=1, \dots, N_{v_2}$ will be given below.  
From the components on $\hat{\bf f}$, since we get a diagonal matrix, 
we have 
\begin{equation}
\label{f2D_fourier}
\hat{f}_{k_x,k_y,j_1,j_2}(t) = e^{-i  (t-t^n)(k_x v_{1, j_1} + k_y v_{2, j_2})}  f_{k_x,k_y,j_1,j_2}^n
\end{equation}
which can be inserted in the Maxwell part to compute the vectors $({}_e{\cal B}, {}_e{\cal E}_{2}, {}_e{\cal E}_{1})_{k_x,k_y,\star, j_2}$.  
To do so, we consider the vector $V(t)=(\hat{B}, \hat{E}_2, \hat{E}_1)^T(t)$ 
which solves the following ODE 
$$
\partial_t V = {\cal J} V + R \mbox{ with } 
R(t)=\left( 
 \begin{array}{lll} 
 0\\ 
 \sum_{j_1,j_2} v_{2, j_2} \hat{f}_{k_x,k_y,j_1,j_2}(t)
 \Delta v_1\Delta v_2\\
\sum_{j_1,j_2} v_{1, j_1} \hat{f}_{k_x,k_y,j_1,j_2}(t)
\Delta v_1\Delta v_2
\end{array} 
 \right) 
\mbox{ and } 
{\cal J} = \left( 
 \begin{array}{lllllllccc}
0 & -ik_x & ik_y\\
-ik_x & 0 & 0 \\
ik_y & 0 &0
\end{array} 
 \right). 
 $$
The variation of constant formula gives 
$$
V(t^{n+1}) = e^{\Delta t {\cal J}}V(t^{n}) + \int_{t^n}^{t^{n+1}} e^{(t^{n+1}-t) {\cal J}} R(t) dt.  
$$
The calculations for the integral term 
involve the following integral term 
\begin{eqnarray*}
{\cal I}_1(\hat{\bf f}^n)  &=&  \int_{t^n}^{t^{n+1}} \Big(\frac{-ik_x\sin(|k|(t^{n+1}-t))}{|k|}R_2(t) +\frac{i k_y\sin(|k|(t^{n+1}-t)))}{|k|} R_3(t)\Big) dt \nonumber\\
&=& \sum_{j_1,j_2} \Big(\frac{(-ik_x v_{2,j_2} + ik_y v_{1,j_1})}{|k|} {\cal C}  \Big) \hat{f}_{k_x,k_y,j_1,j_2}^n \Delta v_1 \Delta v_2, \nonumber\\
&=& \sum_{j_1,j_2} {}_e{\cal B}_{k_x,k_y,j_1,j_2} \hat{f}_{k_x,k_y,j_1,j_2}^n, 
\end{eqnarray*}
\begin{eqnarray*}
{\cal I}_2(\hat{\bf f}^n)  &=&  \int_{t^n}^{t^{n+1}} \Big(\frac{k^2_x\cos(|k|(t^{n+1}-t)) + k_y^2}{|k|^2}R_2(t) +\frac{k_x k_y(1-\cos(|k|(t^{n+1}-t)))}{|k|^2} R_3(t)\Big) dt \nonumber\\
&=& \sum_{j_1,j_2} \Big(v_{2,j_2} \frac{k^2_x {\cal A}  + k_y^2 {\cal B}}{|k|^2} + v_{1,j_1} \frac{k_x k_y({\cal B}-{\cal A})}{|k|^2} \Big) \hat{f}_{k_x,k_y,j_1,j_2}^n \Delta v_1 \Delta v_2, \nonumber\\
&=& \sum_{j_1,j_2} {}_e{\cal E}_{2,k_x,k_y,j_1,j_2} \hat{f}_{k_x,k_y,j_1,j_2}^n, 
\end{eqnarray*}
\begin{eqnarray*}
{\cal I}_3(\hat{\bf f}^n) &=&  \int_{t^n}^{t^{n+1}} \Big( \frac{k_x k_y(1-\cos(|k|(t^{n+1}-t)))}{|k|^2} R_2(t) + \frac{k^2_y\cos(|k|(t^{n+1}-t)) + k_x^2}{|k|^2}R_3(t)  \Big) dt \nonumber\\
&=& \sum_{j_1, j_2} \Big( v_{2,j_2} \frac{k_x k_y ({\cal B}-{\cal A})}{|k|^2} + v_{1,j_1} \frac{k_y^2 {\cal A} + k^2_x {\cal B}}{|k|^2}\Big) \hat{f}_{k_x,k_y,j_1,j_2}^n\Delta v_1 \Delta v_2,  \nonumber\\
&=& \sum_{j_1,j_2} {}_e{\cal E}_{1,k_x,k_y,j_1,j_2} \hat{f}_{k_x,k_y,j_1,j_2}^n,  
\end{eqnarray*}  
where the time integrals are  
\begin{eqnarray*}
{\cal A} &=&  \int_{t^n}^{t^{n+1}} \cos(|k|(t^{n+1}-t)) e^{-i  (t-t^n)(k\cdot v_j)} dt,\nonumber\\
{\cal B} &=& \int_{t^n}^{t^{n+1}} e^{-i  (t-t^n)(k\cdot v_j)} dt = \frac{1}{i k\cdot v_j} (1-e^{-i \Delta t k\cdot v_{j}}), \nonumber\\
{\cal C} &=&  \int_{t^n}^{t^{n+1}} \sin(|k|(t^{n+1}-t)) e^{-i  (t-t^n)(k\cdot v_j)} dt,
\end{eqnarray*}
with $k\cdot v_j = k_x v_{1,j_1} + k_yv_{2,j_2}$. 
To check the conservation of the Poisson equation, one focuses on the equations on $\hat{E}_2$ and $\hat{E}_1$ only. 
Thanks to the above calculations, we can write down the update of $\hat{E}_2, \hat{E}_1$ using the first order Lawson scheme 
\begin{eqnarray*}
\hat{E}_2^{n+1} &=& (e^{\Delta t{\cal J}} V^n)_2 + {\cal I}_2 (\hat{\bf f}^n-\Delta t \widehat{({\cal F}{\cal D}{\bf f}^n)}), \nonumber\\
\hat{E}_1^{n+1} &=& (e^{\Delta t{\cal J}} V^n)_3 + {\cal I}_3 (\hat{\bf f}^n-\Delta t \widehat{({\cal F}{\cal D}{\bf f}^n)}).  
\end{eqnarray*}
Thus, it remains to compute $ik_x \hat{E}_1^{n+1}+ik_y\hat{E}_2^{n+1}$ 
using the last relations. First, one can check easily that  
\begin{eqnarray*}
ik_y(e^{\Delta t{\cal J}} V^n)_2 + ik_x (e^{\Delta t{\cal J}} V^n)_3 &=& ik_y \Big( \frac{-ik_x \sin(|k|\Delta t)}{|k|}\hat{B}^n + \frac{k^2_x\cos(|k|\Delta t) + k_y^2}{|k|^2}\hat{E}_2^n + \frac{k_x k_y(1-\cos(|k|\Delta t))}{|k|^2} \hat{E}_1^n\Big) \nonumber\\
&&+ ik_x \Big( \frac{ik_y \sin(|k|\Delta t)}{|k|}\hat{B}^n +\frac{k_x k_y(1-\cos(|k|\Delta t))}{|k|^2} \hat{E}_2^n + \frac{k^2_y\cos(|k|\Delta t) + k_x^2}{|k|^2}\hat{E}_1^n \Big)\nonumber\\
&=& i k_y \hat{E}_2^{n}+i k_x \hat{E}_1^{n}. 
\end{eqnarray*}
Then, we have 
\begin{eqnarray*}
i k_y \hat{E}_2^{n+1}+i k_x \hat{E}_1^{n+1}  &=& i k_y \hat{E}_2^{n}+i k_x \hat{E}_1^{n} + ik_y {\cal I}_2 (\hat{f}^n-\Delta t \widehat{({\cal F}{\cal D}f^n)})+ ik_x {\cal I}_3 (\hat{f}^n-\Delta t \widehat{({\cal F}{\cal D}f^n)})\nonumber\\
&=& i k_y \hat{E}_2^{n}+i k_x \hat{E}_1^{n} + ik_y \sum_{j_1,j_2} v_{2,j_2} (\hat{f}^n-\Delta t \widehat{({\cal F}{\cal D}f^n)}) {\cal B} \Delta v_2\Delta v_1 \nonumber\\
&& + ik_x\sum_{j_1,j_2} v_{1,j_1}  (\hat{f}^n-\Delta t \widehat{({\cal F}{\cal D}f^n)}) {\cal B} \Delta v_2\Delta v_1 \nonumber\\
&=& i k_y \hat{E}_2^{n}+i k_x \hat{E}_1^{n}  + \Delta v_2\Delta v_1  
\sum_{j_1,j_2} (ik\cdot v_j)  (\hat{f}^n-\Delta t \widehat{({\cal F}{\cal D}f^n)}) \frac{1}{ik\cdot v_j}(1-e^{-i\Delta t k\cdot v_j})\nonumber\\ 
&=& i k_y \hat{E}_2^{n}+i k_x \hat{E}_1^{n}  + \Delta v_2\Delta v_1  
\Big[\sum_{j_1,j_2} (1-e^{-i\Delta t k\cdot v_j}) f^n  \nonumber\\
&&-\Delta t \sum_{j_1,j_2}\widehat{({\cal F}{\cal D}f^n  }) + \Delta t \sum_{j_1,j_2} e^{-i\Delta t k\cdot v_j} \widehat{({\cal F}{\cal D}f^n  }) \Big] \nonumber\\
&=& i k_y \hat{E}_2^{n}+i k_x \hat{E}_1^{n}  + \Delta v_2\Delta v_1 
\Big[\sum_{j_1,j_2} (1-e^{-i\Delta t k\cdot v_j}) \hat{f}^n   +  \sum_{j_1,j_2} e^{-i\Delta t k\cdot v_j} \hat{f}^n - \hat{f}^{n+1}   \Big] \nonumber\\
&=& i k_y \hat{E}_2^{n}+i k_x \hat{E}_1^{n}  + \Delta v_2\Delta v_1 
\sum_{j_1,j_2} (\hat{f}^n - \hat{f}^{n+1}), 
\end{eqnarray*}
where we used the update for $f$: $\hat{f}^{n+1} = e^{-i\Delta t k\cdot v_j} \hat{f}^n -\Delta t e^{-i\Delta t k\cdot v_j} \widehat{({\cal F}{\cal D}f^n})$ and the conservation property of the discrete operator ${\cal D}$. Then, if the Poisson equation is satisfied at iteration $n$, it is propagated to the next iteration, which concludes the proof. 
\end{proof}

\section{Numerical experiments}\label{sec:numerialtests}

    In this section, we perform numerical experiments for linear transport problems and Vlasov equations. 
    First, we study the different order of convergence on a linear problem and then, we present some numerical results of the  exponential DG solutions for Vlasov equations in 1dx-1dv and 1dx-2dv cases.

\subsection{2D linear passive-transport problems}


We consider the following two-dimension linear transport equation 
\begin{equation}
\partial_t u + \partial_x u + \partial_v u =0, \ (x,v)\in[0,2\pi]^2
\end{equation}
with the initial condition $u(x,v,0)=\sin(x+v)$ and periodic boundary condition. The exact solution is $u(x,v,t)=\sin(x+v-2t)$ which enables us to check the different order of convergence. Indeed, for a Lawson scheme based on a underlying Runge-Kutta method RK($m,s)$  (order $m$, $s$ stages, a DG space approximation with $P_k$ and a finite difference approximation in $v$ of order $4$ (which means ${\cal D}$ 
is chosen as a 4th order centered finite difference operator CD4), we expect the following estimate  
$$
\|{\bf u}(t^n)-{\bf u}_h^n\|_{L^2}\leq C (\Delta x^{k}+\Delta v^{4}+\Delta t^{m})
$$
Here we use the 3rd order Lawson-RK method for the time discretization, with a final time $T=1$, and consider different parameters to test  the convergence rates in $t, x$ and $v$. We firstly take $N_v=320,\Delta t=0.01$, and consider different mesh size $N_x$ to check the convergence rate of DG in x direction for both central and upwind fluxes (in this linear case, upwind fluxes can be considereed easily). Table \ref{2Dlinearconstanttestorderinx} shows the $L^\infty$ and $L^2$ errors, the associated orders of convergence for DG-$P^k$ for $k = 1, 2$ in $x-$ direction. The optimal convergence rate for DG is clearly obtained. In particular, the sub-optimal and optimal rates are observed 
according to the choice of the flux and to the oddness of $k$, as discussed in \cite{liu_shu_DG_convergence}. 
Then, we study the convergence in $v-$ direction. We take $k=5, N_x=32, 
\Delta t=0.01$, and consider different mesh size $N_v$ to check the convergence rate of the fourth order approximation of ${\cal D}$ (CD4). Table \ref{2Dlinearconstanttestorderinv} shows the expected convergence (note that only central fluxes are considered in this case). Finally, to check time accuracy, we take $k=5, N_x=16, 
N_v=32$ and $\Delta t=0.0001$ to compute a reference solution. Then we get the error table for different time step sizes $\Delta t$. From Table \ref{2Dlinearconstanttestorderintime}, the expected 3rd order convergence is observed for all cases.
\begin{table}[!hbp]
\centering
\begin{tabular}{|c|c|c|c|c|c|c|c|c|c|}
\hline
\multirow{2}{*}{} & \multirow{2}{*}{$N_x$} & \multicolumn{4}{|c|}{central flux}& \multicolumn{4}{|c|}{upwind flux} \\
\cline{3-10}
& & $L^\infty$-error & order & $L^2$-error & order & $L^\infty$-error & order & $L^2$-error & order  \\
\hline
\multirow{4}{*}{ $P^1$} & $10$ & 1.62E-01 & - & 3.76E-01 & - & 4.78E-02 & - & 9.51E-02 & - \\
& $20$ & 7.66E-02 & 1.08 & 1.86E-01 & 1.02 & 1.27E-02 & 1.91 & 2.42E-02 & 1.98\\
& $40$ & 3.70E-02 & 1.05 & 9.24E-02 & 1.01 & 3.25E-03 & 1.97 & 6.07E-03 & 1.99\\
& $80$ & 1.82E-02 & 1.03 & 4.60E-02 & 1.00 & 8.23E-04 & 1.99 & 1.52E-03 & 2.00\\
& $160$ & 9.01E-03 & 1.01 & 2.30E-02 & 1.00 & 2.07E-04 & 1.99 & 3.79E-04 & 2.00\\
\hline
\multirow{4}{*}{ $P^2$} & $10$ & 2.29E-03 & - & 3.76E-03 & - & 2.52E-03 & - & 4.67E-03 & - \\
& $20$ & 2.65E-04 & 3.11 & 4.53E-04 & 3.05 & 3.07E-04 & 3.04 & 5.83E-04 & 3.00 \\
& $40$ & 3.25E-05 & 3.02 & 5.63E-05 & 3.01 & 3.84E-05 & 3.00 & 7.29E-05 & 3.00 \\
& $80$ & 4.05E-06 & 3.01 & 7.03E-06 & 3.00 & 4.79E-06 & 3.00 & 9.11E-06 & 3.00\\
& $160$ & 5.12E-07 & 2.98 & 8.99E-07 & 2.97 & 5.95E-07 & 3.00 & 1.16E-06 & 2.98\\
\hline
\end{tabular}
\caption{Linear transport equation: $L^\infty$ and $L^2$-norm space errors of the Lawson-DG scheme with  $P^1$, $P^2$ (CD4 in velocity ($N_v=320$) and RK(3,3) in time ($\Delta t=0.01$)). }\label{2Dlinearconstanttestorderinx}
\end{table}

\begin{table}[!hbp]
\centering
\begin{tabular}{|c|c|c|c|c|c|}
\hline
\multirow{2}{*}{} & \multirow{2}{*}{$N_v$} & \multicolumn{4}{|c|}{central flux} \\
\cline{3-6}
& & $L^\infty$-error & order & $L^2$-error & order \\
\hline
\multirow{4}{*}{ CD4 } & $8$ & 1.18E-02 & - & 5.24E-02 & - \\
& $16$ & 7.78E-04 & 3.92 & 3.50E-03 & 3.92 \\
& $32$ & 4.93E-05 & 3.98 & 2.19E-04 & 3.98 \\
& $64$ & 3.09E-06 & 4.00 & 1.37-05 & 4.00 \\
& $128$ & 1.98E-07 & 3.97 & 8.78E-07 & 3.97 \\
\hline
\end{tabular}
\caption{Linear transport equation: $L^\infty$ and $L^2$-norm velocity errors of the Lawson-DG scheme with CD4 (DG-$P^5$ in space ($N_x=32$) and RK(3,3) ($\Delta t=0.01$)).}
\label{2Dlinearconstanttestorderinv}
\end{table}

\begin{table}[!hbp]
\centering
\begin{tabular}{|c|c|c|c|c|c|}
\hline
\multirow{2}{*}{} & \multirow{2}{*}{$\Delta t$} & \multicolumn{4}{|c|}{central flux} \\
\cline{3-6}
& & $L^\infty$-error & order & $L^2$-error & order \\
\hline
\multirow{4}{*}{ RK(3,3)} & $0.1000$ & 4.17E-05 & - & 1.85E-04 & - \\
& $0.0500$ & 5.21E-06 & 2.99 & 2.31E-05 & 2.99 \\
& $0.0250$ & 6.51E-07 & 3.00 & 2.89E-06 & 3.00 \\
& $0.0125$ & 8.14E-08 & 3.00 & 3.61E-07 & 3.00 \\
& $0.00625$ & 1.02E-08 & 3.00 & 4.52E-08 & 3.00 \\
\hline
\end{tabular}
\caption{Linear transport equation: $L^\infty$ and $L^2$-norm time  errors of the Lawson-DG scheme with RK(3,3) (DG-$P^5$ in space ($N_x=16$) and CD4 ($N_v=32$)).}
\label{2Dlinearconstanttestorderintime}
\end{table}

\subsection{Vlasov-Amp\`ere equation} 
We firstly consider the following initial condition for Landau damping
$$
f_0(x,v)=\frac{1}{\sqrt{2\pi}}e^{-v^2/2}(1+\alpha \cos(kx)),
$$
where we take $x\in [0,2\pi/k],\ k=0.5,\ v\in [-9,9]$ and $\alpha=10^{-3}$. Here we still use DG method for space discretization,  the finite difference method $CD4$ in $v$ direction and the 3rd  Lawson-RK method for time discretization (see Section \ref{DG_VA}). 
The numerical parameters are chosen as follows: $\Delta t=0.1$, $N_x=31(P^2), N_v=121$. 

In Figure \ref{VPlandau}, the time evolution of the electric energy 
$\|E(t)\|_{L^2}$ is displayed  in semi-log scale (with the corresponding damping rate in red) 
and the deviation of the total energy ${\cal H}(t)-{\cal H}(0)$ 
with ${\cal H}(t) =\int v^2 f(t, x, v) dxdv + \int E^2(t, x) dx$. 
The expected behaviors (correct damping rate and good energy conservation) are recovered. 
\begin{figure}[!ht]
	\centering
 \includegraphics[height=50mm]{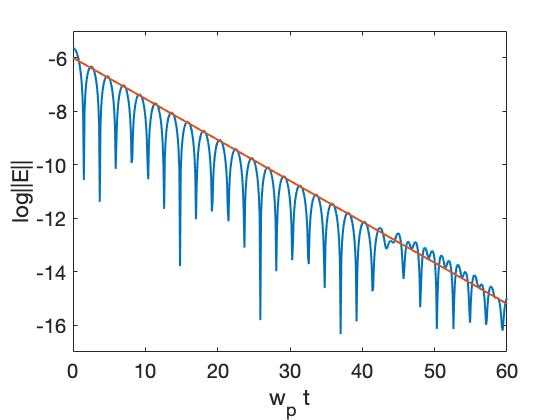}
 \includegraphics[height=50mm]{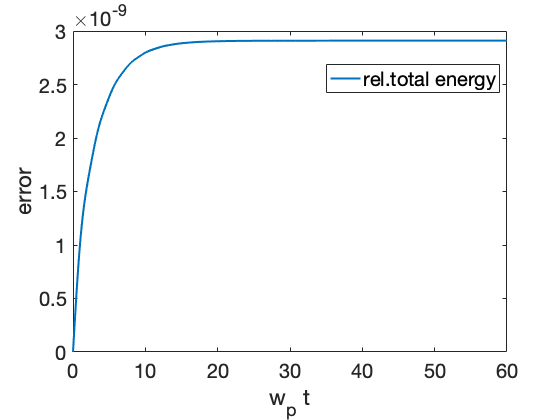}
	\caption{Vlasov-Amp\`ere equation (Landau damping): time evolution of the electric energy in semi-log scale (left) and of the deviation of the total energy (right). Lawson-DG RK(3,3) and $P^2$  ($\Delta t=0.1$, $N_x=31(P^2), N_v=121$). }
	\label{VPlandau}
\end{figure}

We consider a second test called the two stream instability test with the initial condition
$$
f_0(x,v)=\frac{1}{\sqrt{2\pi}}v^2e^{-v^2/2}(1+\alpha \cos(kx)),
$$
for which the same physical and numerical parameters as previously are  kept except the final time which is $T=300$. In Figures \ref{VPtwostream}, we plot the time evolution of the 
electric energy in semi-log scale (and the corresponding instability rate in red) and the deviation of the total energy. 
For this test, a linear instability is first observed (up to $t\approx 30$) during which a vortex 
in phase space is created (see \ref{VPtwostream}), and it is followed by a nonlinear phase.  
These two behaviors are well reproduced by the scheme even if the mesh is quite coarse. In partiular, even if the vortex is well captured, 
we can observe spurious oscillations due to the use of central schemes. Note that the Poisson equation is satisfied in both cases up to machine accuracy. 

\begin{figure}[!ht]
	\centering
 \includegraphics[height=50mm]{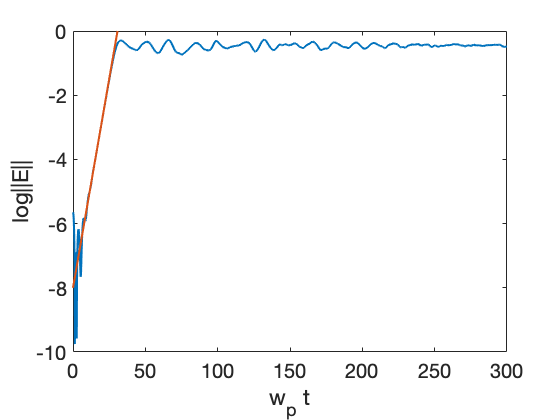}
\includegraphics[height=50mm]{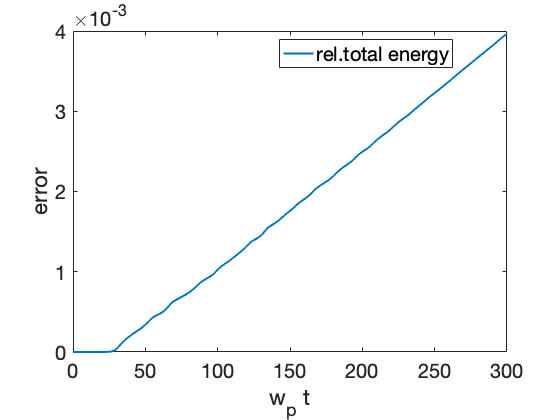}
\includegraphics[height=50mm]{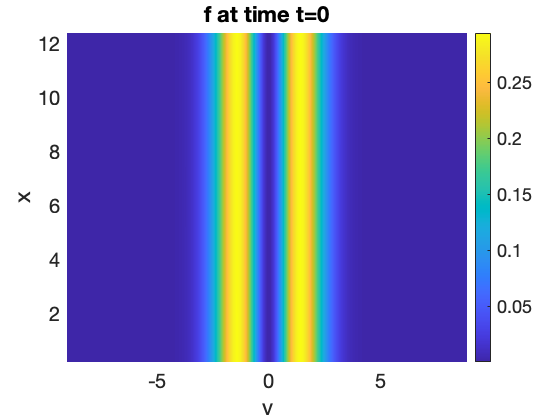}
\includegraphics[height=50mm]{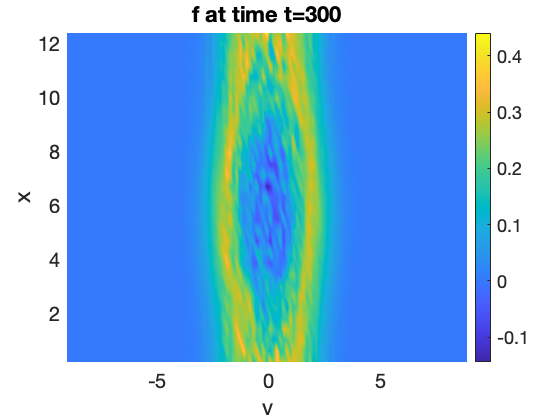}
	\caption{Vlasov-Amp\`ere equation (two stream instability): time evolution of the electric energy  in semi-log scale (top left) and of the deviation of the total energy (top right), snapshot of $f(t=0)$ (bottom left) 
 and snapshot of $f(t=300)$ (bottom right). 
 Lawson-DG RK(3,3) and $P^2$ ($\Delta t=0.1$, $N_x=31(P^2), N_v=121$).}
	\label{VPtwostream}
\end{figure}

\subsection{Vlasov-Maxwell equations 1dx-2dv} 
We consider the Weibel instability \cite{weibel1959spontaneously} by consideering the Vlasov-Maxwell 1dx-2dv model studied in Section \ref{VM1dx-2dv} with the initial distribution
and fields are of the form
$$
f(t=0,x,v_1,v_2)=\frac{1}{2\pi \sigma_1\sigma_2}\exp\left(-\frac{1}{2}\left(\frac{v_1^2}{\sigma_1^2}+\frac{v_2^2}{\sigma_2^2}\right)\right)(1+\alpha \cos(k_wx)),\ x\in [0,2\pi/k),
$$
$$
B(t=0x)=\beta \cos(kx), \;\;\; E_2(t=0,x)=0,
$$
and $E_1(x,t=0)$ is imposed from the Poisson equation. We choose the parameters $\sigma_1=0.02/\sqrt{2},\ \sigma_2=\sqrt{12}\sigma_1,\ k=1.25,\ \alpha=0,\ 
\beta=-10^{-4}$ for our test, which gives a growth rate of 0.02784 by solving the dispersion relation (see  Weibel \cite{weibel1959spontaneously}). 
For the numerical simulations up to  a final time $T=500$, we still use DG method for space discretization in $x-$ direction, finite difference method with $CD4$ in $v$ direction and  Lawson-RK(3,3) method for time discretization (Lawson-RK(3,3)-DG CD4) and consider $\Delta t=0.5$, $N_x=21(P^2), N_{v_1}=N_{v_2}=44$. For comparison, we also consider Fourier method for space discretization in $x-$ direction, finite difference with a third order upwind (UP3) in  $v$ direction and  Lawson-RK(3,3) method in time (Lawson-RK(3,3)-Fourier UP3) with $\Delta t=0.5$, $N_x=64, N_{v_1}=N_{v_2}=88$. 

In Figure \ref{VMWeibel instability}, we show the time evolution of 
the electromagnetic energies $\|B(t)\|_{L^2}$, $\|E_y(t)\|_{L^2}$, $\|E_x(t)\|_{L^2}$ (in semi-log scale) obtained by the two methods. 
First, we can observe that the theoretical growth rate is in very good agreement with the two numerical solution. Second, the two methods are very close up to time $t\approx 300$ (which corresponds to the end of the linear phase) and slightly differs for larger times. 
We also show the evolution of the relative total energy $|{\cal H}(t) - {\cal H}(0)|/{\cal H}(0)$ in Figure \ref{VMWeibelinstabilitytotalenergy} for Lawson-RK(3,3)-DG CD4 scheme without (left) and with (right) the energy correction step presented in \cite{boutin2022modified} in the Vlasov-Amp\`ere case.  
This projection approach enables to modify the unknown by a suitable 
coefficient which is of order the scheme so that the total energy is preserved almost up to machine error $10^{-12}$.

\begin{figure}[!ht]
 \centering
 \includegraphics[height=50mm]{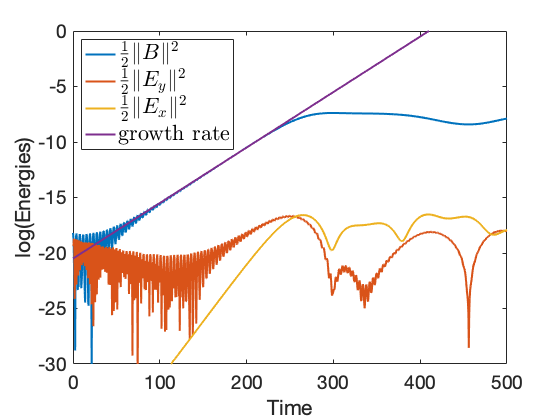}
 \includegraphics[height=50mm]{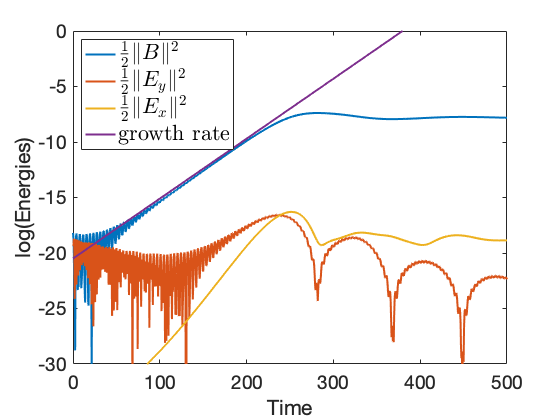}
	\caption{Vlasov-Maxwell equation (Weibel instability): time evolution of the electromagnetic energies in semi-log scale  together with
the analytic growth rate. Left: Lawson-RK(3,3)-DG CD4. Right: Lawson-RK(3,3)-Fourier UP3.}
	\label{VMWeibel instability}
\end{figure}

\begin{figure}[!ht]

 \centering
 \includegraphics[height=50mm]{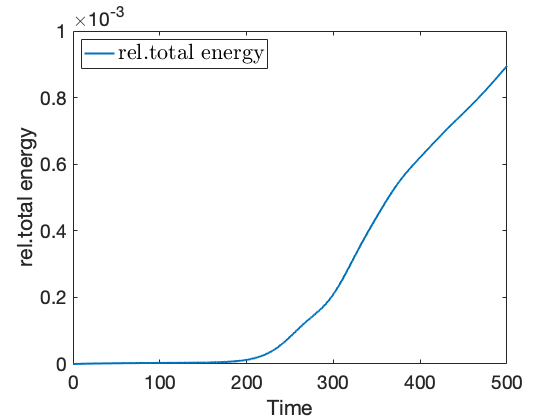}
 \includegraphics[height=50mm]{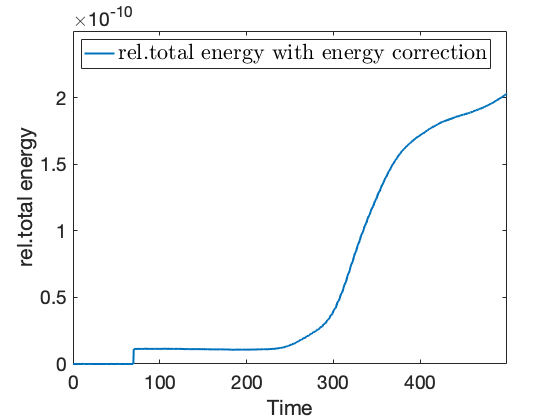}
	\caption{Vlasov-Maxwell equation (Weibel instability): time evolution the relative total energy of Lawson-RK(3,3)-DG CD4  without (left) and with (right) energy correction step.}
	\label{VMWeibelinstabilitytotalenergy}
\end{figure}


The second test for Vlasov-Maxwell equation we considered is the streaming Weibel instability  \cite{califano1997spatial,cheng2014discontinuous} for which the  initial condition is 
$$
f(t=0,x,v_1,v_2)=\frac{1}{2\pi \sigma^2}\exp\left(-\frac{v_1^2}{2\sigma^2}\right)\left(\delta \exp\left(-\frac{(v_2-v_{0,1})^2}{2\sigma^2}\right)+(1-\delta) \exp\left(-\frac{(v_2-v_{0,2})^2}{2\sigma^2}\right)\right), 
$$
$$
B(t=0, x,t=0)=\beta \cos(kx), \;\; E_2(t=0, x)=0,
$$
 and $E_1(t=0,x)=0$ from the Poisson equation. 
 We choose the parameters $\sigma=0.1/\sqrt{2},\ k=0.2,\ 
\beta=-10^{-3},\ v_{0,1}=0.5,\ v_{0,2}=-0.1$ and $\delta=1/6$ for our test. The growth rate of $E_2$ is 0.03 \cite{califano1997spatial}. For the two schemes Lawson-RK(3,3)-DG CD4 and Lawson-RK(3,3)-Fourier UP3, we take the same parameters as in the previous case but $\Delta t=0.1$ for stability reasons. 
We show the result in Figure \ref{VMstreamingWeibelinstability} in which the time evolution of the $L^2$ norm of the electromagnetic fields are displayed. 
First, we observe a good agreement with the theoretical growth rate  for these two schemes and some deviations in the nonlinear phase can be observed. 
For this case, we also consider the correction on the total energy 
and plot the time history of the relative total energy  in Figure \ref{VMstreamingWeibelinstabilitytotalenergy} for Lawson-RK(3,3)-DG CD4   without (left) and with (right) energy correction step. The total energy can be well preserved for exponential DG approximations if it is stable, and the relative error is greatly improved with energy correction step without affecting the accuracy of the scheme. 
Note that the Poisson equation is satisfied in both cases up to machine accuracy. 

\begin{figure}[!ht]
 \centering
 \includegraphics[height=50mm]{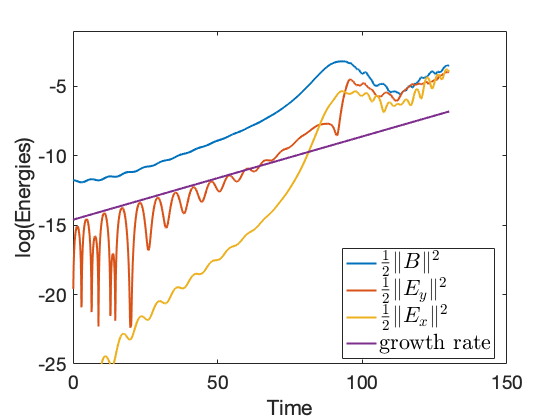}
  \includegraphics[height=50mm]{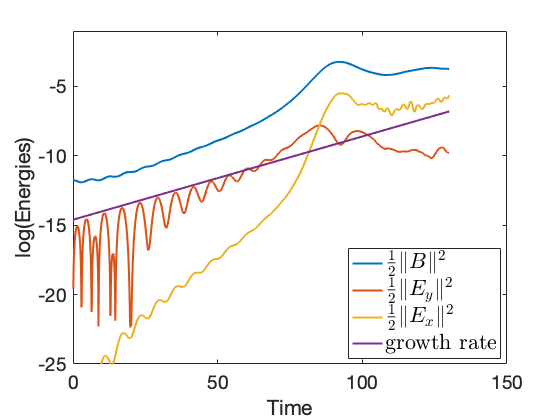}
	\caption{Vlasov-Maxwell equation (streaming Weibel instability): time evolution of the electromagnetic energies in semi-log scale  together with
the analytic growth rate. Left: Lawson-RK(3,3)-DG CD4. Right: Lawson-RK(3,3)-Fourier UP3.}
	\label{VMstreamingWeibelinstability}
\end{figure}

\begin{figure}[!ht]

 \centering
  \includegraphics[height=50mm]{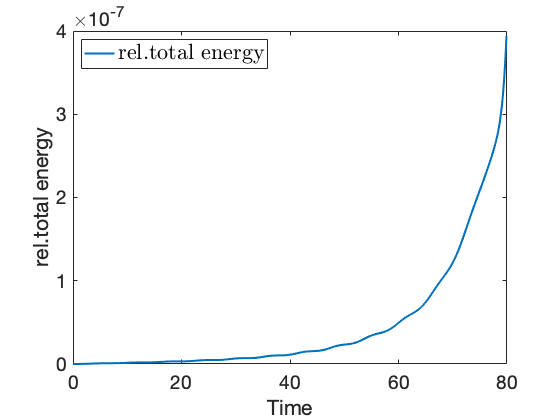}
 \includegraphics[height=50mm]{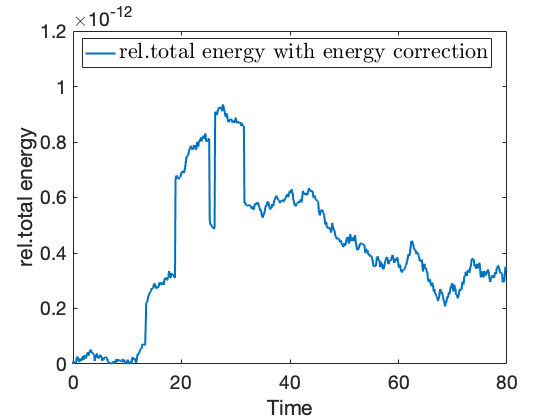}

	\caption{Vlasov-Maxwell equation (streaming Weibel instability): time evolution the relative total energy of Lawson-RK(3,3)-DG CD4  without (left) and with (right) energy correction step.}
	\label{VMstreamingWeibelinstabilitytotalenergy}
\end{figure}

\section{Conclusion}\label{sec:conclusion}
In this work, we constructed and implemented a new exponential DG  method for Vlasov equation, extending the previous works \cite{crouseilles2020exponential, crestetto2022comparison, boutin2022modified} on this topic where Fourier method in space were used. These methods allow to derive high order accuracy in time, space and velocity, still ensuring stability without the restrictive CFL type constraint 
coming from the linear part. Moreover, a discrete Poisson equation 
is satisfied and a projection technique enables to preserve the total energy. 
The extension to DG turns out to be 
an interesting alternative compared to previous approach based on Fourier which is restricted to cartesian domains with periodic boundary conditions. The approach only involves the calculation of exponential 
of DG-matrices of size $(k+1)N$ with $k$ the DG degree and $N$ the number of points in the space direction and we observe that thanks to 
the specific structure, this is also true in multi-dimensions. 

One interesting extension is to perform an efficient implementation of the method in two-dimension in space to exploit the Kronecker structure. We also plan to investigate the extension of this approach to  problems involving boundary conditions, for which monotone fluxes 
are more appropriate but requires to study the stability of the scheme. 


\appendix
 \section{Appendix A: The DG matrix construction
}
\label{appendixB1}
The goal of this appendix is to give some details on the construction of the DG-matrix \eqref{dg_matrix} which is at the core of our scheme.  
To do so, we consider the 1dx transport equation \eqref{linear_transport} with coefficient $a=1$, for which 
we wrote in \eqref{eq:semi-discreteDGschemefortransport} the semi-discrete DG scheme. 
To get the DG matrix, we consider $\psi=\xi^\ell_j$ and 
$u_h(t, x)=\sum_{m=0}^k u^m_j(t) \xi^m_j(x)$  in 
\eqref{eq:semi-discreteDGschemefortransport} to get 
\begin{equation}
 \sum_{m=0}^{k}\Big[(\partial_t{u}_{j}^m(t) \xi^m, \xi^\ell)_{I_j}-({u}_{j}^m(t) \xi^m,\partial_x\xi^\ell)_{I_j}\Big]+
 [\{u_h\}\xi^\ell]^{x_{j+1/2}}_{x_{j-1/2}}=0, 
\end{equation}
where $\ell=0,\dots,k, j=1,\dots,N$ and where the central flux 
$\{u_h\}|_{x_{j\pm 1/2}}=(u_h^+(x_{j\pm1/2})+u_h^-(x_{j\pm1/2}))/2$ is considered. Thus we obtain for the boundary term 
\begin{eqnarray*}
[\{u_h\}\xi^\ell]^{x_{j+1/2}}_{x_{j-1/2}} 
&=& \frac{1}{2}\sum_{m=0}^{k}\left[\left({u}_{j}^m(t) \xi^{m, -}(x_{j+1/2})+{u}_{j+1}^m(t) \xi^{m,+}(x_{j+1/2})\right)\xi^{\ell,-}(x_{j+1/2}) \right. \nonumber\\
&& \left.-\left({u}_{j-1}^m(t) \xi^{m,-}(x_{j-1/2})+{u}_{j}^m(t)\xi^{m,+}(x_{j-1/2})\right)\xi^{l,+}(x_{j-1/2})\right]\\
&=& \frac{1}{2}\sum_{m=0}^{k}\left[\left({u}_{j}^m(t) (1/2)^m+{u}_{j+1}^m(t) (-1/2)^m\right)(1/2)^\ell \right. \nonumber\\
&&\left. -\left({u}_{j-1}^m(t) (1/2)^m+{u}_{j}^m(t) (-1/2)^m\right)(-1/2)^\ell\right].
\end{eqnarray*}
We denote $\mathbf{u}_{j}(t)=({u}_{j}^0(t), {u}_j^1(t), ....,{u}_j^k(t))^T, \ j=1,2,...,N,$ 
then we can rewrite the DG discretization as 
\begin{equation}
\begin{aligned}
    \Delta x \left(\begin{array}{llll}
	M&  &  &  \\ 
	 & M&  &  \\
     &  & \ddots & \\
     &  &  & M\\
\end{array}
\right)
\partial_t \left(\begin{array}{llll}
	\mathbf{u}_{1}\\ 
 \mathbf{u}_{2}\\ 
 \vdots\\ 
\mathbf{u}_{N}
\end{array}\right)
&-
\left(\begin{array}{llll}
	D_1& D_2 & \hdots &D_3 \\ 
	D_3 & D_1& D_2 &  \\
     & \ddots & \ddots &D_2 \\
    D_2 & \hdots & D_3 & D_1\\
\end{array}
\right)
\left(\begin{array}{llll}
	\mathbf{u}_{1}\\ 
 \mathbf{u}_{2}\\ 
 \vdots\\ 
\mathbf{u}_{N}
\end{array}\right)
={\bf 0},
\end{aligned}
\end{equation}
where $M, D_i \in\mathbb{M}_{(k+1),(k+1)}(\mathbb{R}), i=1,2, 3$ are given by 
\begin{eqnarray*} 
M_{\ell,m}=\frac{(1/2)^{m+\ell-1}}{m+\ell-1}[1-(-1)^{m+\ell-1}], \!\!\!\!\!\! && (D_2)_{\ell,m}= (-1)^{m}(1/2)^{m+\ell-1},\nonumber\\
(D_1)_{\ell,m}=(1/2)^{m+\ell-2}\Big(\frac{\ell-1}{m+\ell-2}-\frac{1}{2}\Big)[1-(-1)^{m+\ell-2}], \!\!\!\!\!\!  &&   (D_3)_{\ell,m}=(-1)^{\ell-1}(1/2)^{m+\ell-1},  
\end{eqnarray*}
with $(D_1)_{1,1}=0$ by convention. 
Then we have 
 $$\Delta x\mathbfcal{M}\partial_t\mathbf{u}-\mathbfcal{D} \mathbf{u}=0,$$
where $\mathbfcal{M}$ is a block diagonal mass matrix of size $N(k+1)$, $\mathbfcal{D}$ is a block tridiagonal matrix of size $N(k+1)$, and $\mathbf{u}$ is the vector containing the degree of freedom 
$$
\mathbf{u}=(u_1^0, u_1^1, \dots,u_1^k, u_2^0, u_2^1, \dots,u_2^k,\dots,u_N^0, u_N^1, \dots,u_N^k)^T\in\mathbb{R}^{(k+1)N}.
$$
Now we can rewrite the DG scheme as 
\begin{equation}\label{eq:defineDGmattrixA}
    \partial_t \mathbf{u}(t)=A \mathbf{u}(t), \  \ A=(\Delta x \mathbfcal{M})^{-1}\mathbfcal{D},
\end{equation} 
where
$A$ is a block circulant matrix 
$$
A=\frac{1}{\Delta x}\left(\begin{array}{lllll}
	M^{-1}D_1& M^{-1}D_2 & \bf{0} & \hdots &M^{-1}D_3 \\ 
	M^{-1}D_3 & M^{-1}D_1& M^{-1}D_2 & \bf{0}&\hdots  \\
    \bf{0}& \ddots & \ddots & \ddots &M^{-1}D_2 \\
    M^{-1}D_2 & \bf{0}& \hdots & M^{-1}D_3 & M^{-1}D_1
\end{array}
\right) = \frac{1}{\Delta x} A_1, 
$$
or with the \mbox{circblock} notation $A_1=\mbox{circblock}(M^{-1}D_1, M^{-1}D_2, 0, 0, ..., M^{-1}D_3)$. Since $\mathbfcal{D}$ and $\mathbfcal{M}$ are independent of $\Delta x$, so is 
$A_1$ defined by $A=\Delta x^{-1} A_1$.

\section{Appendix B: Stability and error estimate for semi-discrete DG scheme.}
\label{appendixA}
In this appendix, we give some error estimate of the exponential-DG 
scheme for the one dimensional linear advection equation \eqref{linear_transport}. To do so, 
we first define some notations about norms which will be used. 
For a given function $x\mapsto v(x)$, we denote $\Vert v \Vert_j$ and $\Vert v \Vert_{\infty, j}$ as the $L^2$-norm and $L^\infty$-norm of $v$ on $I_j$ ($j=1, \dots, N$) respectively. Moreover,
\begin{equation*}
\begin{aligned}
&\Vert v\Vert=(\sum \limits_j \Vert v \Vert ^2_{j})^{\frac{1}{2}},                                                &\Vert v\Vert_{\infty}=\max\limits_j \Vert v \Vert_{\infty,j}, \\
&\Vert v\Vert_\Gamma^2=\sum_j(\vert v_{j-1/2}^+\vert^2+\vert 
v_{j-\frac{1}{2}}^-\vert^2), 
\end{aligned}
\end{equation*}
where we express the value of $v$ on the left and right limits of the grid point $x_{j+\frac{1}{2}}$ with $v_{j+\frac{1}{2}}^-$ and $v_{j+\frac{1}{2}}^+$ respectively. Define the jump and the mean of $v$ at $x_{j-\frac{1}{2}}$ as $[v]_{j-\frac{1}{2}}=(v_{j-\frac{1}{2}}^+-v_{j-\frac{1}{2}}^-)$ and $\{v\}_{j-\frac{1}{2}}=(v^+_{j-\frac{1}{2}}+v^-_{j-\frac{1}{2}})/2$.

\subsection{Notations for projections and some properties of approximation space}

The inverse properties of the finite space $V_h$ will be used.
\begin{lemma}\label{lemma:the derivative and point error estimate in Vh}
When the mesh is regular, $\forall v\in V_h,\ \exists\ C>0,\ s.t.$
\begin{equation}
h^2\Vert\partial_xv\Vert^2+h\Vert v\Vert_\Gamma^2\leqslant C\Vert v\Vert^2\label{eqn:2.19},
\end{equation}
where the positive constant $C$ is independent of $h$ and $v$.
\end{lemma}

Define the $L^2$-projection $P_k$
of $u$ into $V_h$ as follows:
\begin{equation*}
\begin{aligned}
  (P_k u, v_h)&=(u, v_h),  \forall v_h\in V_h \mbox{ with } v_h(\chi_j(.,t))\in P^{k}([-1,1]).
\end{aligned}
\end{equation*}
The following lemma states the error of these projections {\cite{ciarlet2002finite}}.
\begin{lemma}\label{lemma:the projection error estimate in Vh}
Let 
$P_h ^\perp q=q-P_kq$ is the projection error. For any smooth function $q(x)$, $\exists\ c>0,$ such that
\begin{equation}
\Vert P_h^\perp q\Vert_D+h\Vert\partial_x(P_h^\perp q)\Vert_D+h^{\frac{1}{2}}\Vert P_h^\perp q\Vert_{\infty,D}\leqslant ch^{k+1}\vert q\vert_{k+1,D}, \label{eqn:2.20}
\end{equation}
\begin{equation}
\Vert P_h^\perp q\Vert_\Gamma \leqslant ch^{k+\frac{1}{2}}\Vert\partial_x^{k+1}q\Vert, \label{eqn:2.21}
\end{equation}
where the positive constant $c$ is not dependent on $h$, solely depending on $q$, and $D$ may be $\Omega$ or $I_j$.
\end{lemma}

Furthermore, to avoid confusion with different constants, we denote a generic positive constant by $C$, which is independent of the numerical solution and the mesh size for our problem. But, the constant may dependent on the exact solution and may have a different value in each occurrence. Moreover, for problems considered in this paper, the exact solution is assumed to be smooth with periodic or compactly supported boundary condition. Therefore, the exact solution is always bounded. 

We state the $L^2$ stability and $L^2$-norm error estimate for the scheme and also give their proof.
\begin{theorem}\label{$theorem: L^2$-stability}
For semi-discrete DG scheme \eqref{eq:semi-discreteDGschemefortransport} with central flux, we have the $L^2$-stability:
$$\frac{d}{dt}||u_h||_{L^2}^2=0.$$
\end{theorem}
\begin{proof}
Take the test function $\psi=u_h$ in the semi-discrete scheme \eqref{eq:semi-discreteDGschemefortransport}, we have
\begin{equation}
\begin{aligned}
    \frac{1}{2}\frac{d}{dt} \int_{{I}_j}u_h^2dx &=-  a\{u_h\}|_{{x}_{j+\frac12} }  u_h^-|_{{x}_{j+\frac12} } +   a\{u_h\}|_{{x}_{j-\frac12} } u_h^+|_{{x}_{j-\frac12} }  + \int_{{I}_j}a u_h (u_h)_xdx,\\
    &=-  a\{u_h\}|_{{x}_{j+\frac12} }  u_h^-|_{{x}_{j+\frac12} } +   a\{u_h\}|_{{x}_{j-\frac12} } u_h^+|_{{x}_{j-\frac12} }  + \frac{a}{2}(u_h^2)^-|_{{x}_{j+\frac12} }-\frac{a}{2}(u_h^2)^+|_{{x}_{j-\frac12} }.\\
    &=- \frac{a}{2}u_h^+u_h^-|_{{x}_{j+\frac12} }+\frac{a}{2}u_h^-u_h^+|_{{x}_{j-\frac12} }.
\end{aligned}
\end{equation}
 Sum over $j$ of above equation, the $L^2$-stability follows.
\end{proof}
\begin{theorem}\label{$theorem: L^2$-norm error estimate for DG scheme}
Let $T>0$, $u$ be the exact solution of problem \eqref{linear_transport},
which is sufficiently smooth with bounded derivatives.
Assume $u_h$ is the DG approximation of semi-discrete scheme \eqref{eq:semi-discreteDGschemefortransport} with the central flux and the approximation space $V_h$ is the space consisting of $k$-th piecewise polynomial. Then it holds that
\begin{equation}
\Vert u(T)-u_h(T)\Vert_{L^2} \leqslant Ch^{k} ,
\end{equation}
where C is a positive constant independent on $\Delta x$.
\end{theorem}

\begin{proof}
Denote error as $e_u=u_h-u$.
 Notice that the scheme  \eqref{eq:semi-discreteDGschemefortransport} is still satisfied with $u_h=u$. So, we have the error equation
 \begin{align}
 \int_{{I}_j}(\partial_t e_u\psi)dx =-  a\{e_u\}|_{{x}_{j+\frac12} }  \psi^-|_{{x}_{j+\frac12} } +   a\{e_u\}|_{{x}_{j-\frac12} } \psi^+|_{{x}_{j-\frac12} }   + \int_{{I}_j}a e_u\psi_xdx.
\end{align}
 Define $e_u=u_h-u=(u_h-P_ku)-(u-P_ku)=\widetilde{e_u} -P_k^\perp u$. Then taking $\psi=\widetilde{e_u}$, we have
 \begin{equation}\label{eqn: error equationfor DG scheme}
 \begin{aligned}
   &\int_{{I}_j}\partial_t (\widetilde{e_u})\widetilde{e_u}dx\\
   &=\int_{{I}_j}\partial_t (P_k^\perp u)\widetilde{e_u}dx -  a\{\widetilde{e_u}\}|_{{x}_{j+\frac12} }  \widetilde{e_u}^-|_{{x}_{j+\frac12} } +   a\{\widetilde{e_u}\}|_{{x}_{j-\frac12} } \widetilde{e_u}^+|_{{x}_{j-\frac12} }   + \int_{{I}_j}a \widetilde{e_u}(\widetilde{e_u})_xdx\\
   &+  a\{P_k^\perp u\}|_{{x}_{j+\frac12} }  \widetilde{e_u}^-|_{{x}_{j+\frac12} } -   a\{P_k^\perp u\}|_{{x}_{j-\frac12} } \widetilde{e_u}^+|_{{x}_{j-\frac12} }   - \int_{{I}_j}a P_k^\perp u(\widetilde{e_u})_xdx.
 \end{aligned}
 \end{equation}
By the the definition of the projections and some calculations, the right terms $RHS$ of the error equation \eqref{eqn: error equationfor DG scheme} become
 \begin{equation}
 \begin{aligned}   &RHS=- \frac{a}{2}\widetilde{e_u}^+\widetilde{e_u}^-|_{{x}_{j+\frac12} }+\frac{a}{2}\widetilde{e_u}^-\widetilde{e_u}^+|_{{x}_{j-\frac12} }+  a\{P_k^\perp u\}|_{{x}_{j+\frac12} }  \widetilde{e_u}^-|_{{x}_{j+\frac12} } -   a\{P_k^\perp u\}|_{{x}_{j-\frac12} } \widetilde{e_u}^+|_{{x}_{j-\frac12} }.
 \end{aligned}
 \end{equation}
 Sum over $j$, 
 $$\frac{1}{2}\frac{d}{dt} ||\widetilde{e_u}||^2=\sum_j RHS =-\sum_ja\{P_k^\perp u\}[\widetilde{e_u}]|_{{x}_{j+\frac12}}.$$
 Furthermore, from Holder's inequality, Lemma \eqref{lemma:the derivative and point error estimate in Vh} and Lemma \eqref{lemma:the projection error estimate in Vh}, we have
 \begin{equation}
   \begin{aligned}
     &\frac{1}{2}\frac{d}{dt} ||\widetilde{e_u}||^2\leq \sum_j\left|a\{P_k^\perp u\}[\widetilde{e_u}]|_{{x}_{j+\frac12}}\right |\\
     &\leq c ||P_k^\perp u||_\Gamma ||\widetilde{e_u}||_\Gamma \leq c h^{k+\frac{1}{2}}h^{-\frac{1}{2}}||\widetilde{e_u}||\\
     &\leq c\Vert \widetilde{e_u}\Vert^2+ch^{2k}.
   \end{aligned}
 \end{equation}

Thus by Gronwall's inequality, the conclusion in Theorem \ref{$theorem: L^2$-norm error estimate for DG scheme} follows.
\end{proof}

\section{Appendix C: Ker$(A)$ and projection $\Pi$}
\label{appendixC}
Here, the projection matrix $\Pi$ onto Ker$(A)$ with $A$ the DG-matrix  \eqref{dg_matrix} is discussed. A general expression (for any $k, N$) 
turns out to be difficult and we compute Ker$(A)$ (and the projection $\Pi$) for several practical cases. As mentioned in Remark \ref{remark_eigenA}, there are mainly two cases, according to the 
oddness of $(k+1)N$: $(i)$ if $(k+1)N$ is odd, $0$ is a simple eigenvalue of $A$
and Ker$(A)$=Span$(u_1)$ where $u_1\in\mathbb{R}^{(k+1)N}$ correspond 
constants in the space $P^k$; $(ii)$ if $(k+1)N$ is even, $0$ is a double eigenvalue of $A$ and Ker$(A)$=Span$(u_1, u_2)$ and $u_2\in\mathbb{R}^{(k+1)N}$ has to be determined. 
This second case recalls what happens for the second centered finite differences in which constant vector belongs to the kernel but 
also the sequence $(-1)^j$.

\begin{itemize}
    \item odd case: 
    We can check that 
    $$
    u_1=\Big[1, \underbrace{0,\hdots,0}_{\in \mathbb{R}^k},  1,\underbrace{0,\hdots,0}_{\in \mathbb{R}^k}, \hdots,  1, \underbrace{0,\hdots,0}_{\in \mathbb{R}^k}\Big]^T\in\mathbb{R}^{(k+1)N}.
    $$ is a eigenvector of $A$ associated to the eigenvalue $0$. 
    By $\Pi x=\langle x,u_1\rangle  u_1$, we get the expression of the matrix $\Pi$ 
    $$
    \Pi=[u_1,{\bf 0},\hdots,{\bf 0}, u_1,{\bf 0},\hdots,{\bf 0}, \hdots,  u_1,{\bf 0},\hdots,{\bf 0}\Big]\in\mathbb{M}_{(k+1)N,(k+1)N}(\mathbb{R}), 
    $$ 
    with ${\bf 0}\in\mathbb{R}^{(k+1)N}$. 
\item even case: in addition to $u_1$, we need to find a second  eigenvector to construct $\Pi$. We give below the expression of $u_2$ 
for some $k=1$ to $k=5$  
\begin{itemize}
\item  $k=0$, $u_2=\Big[0, 1,  0,1, \hdots,  0, 1\Big]^T\in\mathbb{R}^{N}.$
\item  $k=1$, $u_2=\Big[\underbrace{0, 1}_{\in\mathbb{R}^{2}}, \underbrace{0, 1}_{\in\mathbb{R}^{2}}, \hdots, \underbrace{0, 1}_{\in\mathbb{R}^{2}}\Big]^T\in\mathbb{R}^{2N}.$
\item  $k=2$, $u_2=\Big[\underbrace{-\frac{1}{6}, 0,1,0,0,-1}_{\in\mathbb{R}^{6}},\hdots,\underbrace{-\frac{1}{6}, 0,1,0,0,-1}_{\in\mathbb{R}^{6}}\Big]^T\in\mathbb{R}^{3N}.$
\item  $k=3$, $u_2=\Big[\underbrace{0,-\frac{3}{20}, 0,1}_{\in\mathbb{R}^{4}},\underbrace{0,-\frac{3}{20}, 0,1}_{\in\mathbb{R}^{4}},\hdots,\underbrace{0,-\frac{3}{20}, 0,1}_{\in\mathbb{R}^{4}}\Big]^T\in\mathbb{R}^{4N}.$
\item  $k=4$, 
$v_2=\Big[\underbrace{-\frac{3}{280},0,\frac{3}{14},0,-1,0,0,-\frac{3}{14},0,1}_{\in\mathbb{R}^{10}},\hdots,\underbrace{-\frac{3}{280},0,\frac{3}{14},0,-1,0,0,-\frac{3}{14},0,1}_{\in\mathbb{R}^{10}}\Big]^T\in\mathbb{R}^{5N}.$
\item case $k=5$, $u_2=\Big[\underbrace{0,\frac{5}{336}, 0,-\frac{5}{18},0,1}_{\in\mathbb{R}^{6}},\underbrace{0,\frac{5}{336}, 0,-\frac{5}{18},0,1}_{\in\mathbb{R}^{6}},\hdots,\underbrace{0,\frac{5}{336}, 0,-\frac{5}{18},0,1}_{\in\mathbb{R}^{6}}\Big]^T\in\mathbb{R}^{6N}.$
\end{itemize}
 In the even case, when $k$ is even, we observe a double pattern which is repeated $N/2$ (since when $k$ is even, $N$ is even to ensure $(k+1)N$ is even). Once we get $(u_1,u_2)$, the formula $\Pi x=\langle x,u_1\rangle u_1+\langle x,u_2\rangle u_2$ enables to get  $\Pi$. 
\end{itemize}

\section{Appendix D: Lawson-Fourier method for Vlasov-Maxwell 1dx-2dv.}
\label{appendixD}
In this appendix, we extend the method presented in \cite{boutin2022modified} to the Vlasov-Maxwell model in 1dx-2dv. 
This approach is compared to the Lawson-DG method in the numerical section \ref{sec:numerialtests}. 

Starting from the Vlasov-Maxwell 1dx-2dv model \eqref{eq:model} 
satisfied by \\ 
$f(t, x, v_1, v_2), E_1(t, x), E_2(t, x), B(t, x)$, 
 with $t\geq 0, x\in [0,L]$ and $(v_1, v_2)\in \mathbb{R}^2$, 
we shall use a Fourier method in the space direction $x$ and we consider a grid in the velocity direction $v_{\ell,j_\ell} = v_{\ell, \min} + j_\ell \Delta v_\ell, \ell=1, 2$. 
Denoting $\hat{f}_{k,j_1,j_2}(t)$ the spatial Fourier coefficients of $f(t, x, v_{j_1}, v_{j_2})$ and $(E_{1,k}, E_{2,k}, B_k)(t)$ the Fourier coefficients of  $(E_1, E_2, B)(t, x)$ then  gives the following semi-discretized scheme for $k=1, \dots, N_x$ ($N_x$ being the number of grid points in $[0, L]$) 
\begin{equation}
\label{eq:model_fourier}
	\left\{
	\begin{aligned}
		&\partial_t \hat{f}_{k,j_1,j_2}+v_{1,j_1}i k \hat{f}_{k,j_1,j_2}+ \widehat{({\cal F} {\cal D} f)}_{k,j_1,j_2} = 0, \;\; \mbox{ with }  {\cal F} =(E_1 + B v_2, E_2  -B v_1  ), \\
		&\partial_t \hat{B}_k =- ik \hat{E}_{2,k}, \\
		&\partial_t \hat{E}_{1,k}=-\sum_{j_1, j_2} v_{1, j_1} \hat{f}_{k,j_1,j_2} \Delta v_1\Delta  v_2,\\
		&\partial_t \hat{E}_{2,k}=-ik \hat{B}_k- \sum_{j_1, j_2} v_{2, j_2} \hat{f}_{k,j_1,j_2} \Delta v_1\Delta v_2,  
	\end{aligned}
\right.
\end{equation} 
with the initial condition $\hat{f}_{k, j_1, j_2}, \hat{B}_k(0), \hat{E}_{1,k}(0),\hat{E}_{2,k}(0)$ satisfying the Poisson equation 
$ik\hat{E}_{1,k}(0) = \sum_{j_1,j_2} \hat{f}_{k,j_1,j_2}(0)\Delta v_1\Delta v_2$ for $k\neq 0$. 
Let denote $U=(\hat{\bf f}, \hat{B}, \hat{E}_2, \hat{E}_1)\in \mathbb{M}_{N_{v_1}N_{v_1}+3,N_{v_1}N_{v_1}+3}(\mathbb{C})$, then the previous system can be rewritten as 
\begin{equation}
\label{ULU}
\partial_t U = L U + N(U), 
\end{equation}
with 
\begin{eqnarray}
U&=& \left( 
 \begin{array}{llll}
\hat{f}_{k,\star,1}\\
\hat{f}_{k,\star,2}\\
\vdots \\
\!\!\! \hat{f}_{k,\star,N_{v_2}}  \!\!\!\!\\
\hat{B}_k\\
\hat{E}_{2,k}\\
\hat{E}_{1,k}
\end{array} 
 \right), \;\;\;\;\;\;   
 N(U)= 
 \left( 
 \begin{array}{llllccc}
-\widehat{({\cal F} {\cal D} f)}_{k,\star,1}\\
-\widehat{({\cal F} {\cal D} f)}_{k,\star,2}\\
\vdots\\
-\widehat{({\cal F} {\cal D} f)}_{k,\star,N_{v_2}}\\
 0 \\
   0 \\
0
\end{array} 
 \right), \;\;\;\; \hat{f}_{k,\star,j_2}\in \mathbb{C}^{N_{v_1}}, 
 \begin{array}{llcc} \forall k&=1,\dots, N_x, \\ 
 \forall j_2 &=1,\dots, N_{v_2},    
 \end{array} 
 \nonumber\\
 \label{L_fourier}
L\!\!\!&=&\!\!\! \left( 
 \begin{array}{lllllllccc}
-ik\mbox{diag}(\vec{v_{1}}) & {\bf 0}_{N_{v_1},N_{v_1}} & {\bf 0}_{N_{v_1},N_{v_1}}\hdots &   {\bf 0}_{N_{v_1},N_{v_1}} &  {\bf 0}_{N_{v_1},1} & \dots & {\bf 0}_{N_{v_1},1}\\
 {\bf 0}_{N_{v_1},N_{v_1}} & -ik\mbox{diag}(\vec{v_{1}}) &  {\bf 0}_{N_{v_1},N_{v_1}}\hdots &   {\bf 0}_{N_{v_1},N_{v_1}} & \vdots &  \vdots & \vdots\\
\vdots & \ddots & \ddots  &  {\bf 0}_{N_{v_1},N_{v_1}}  &  \vdots & \vdots & \vdots\\
 {\bf 0}_{N_{v_1},N_{v_1}} & \hdots  &   {\bf 0}_{N_{v_1},N_{v_1}} & -ik\mbox{diag}(\vec{v_{1}}) &  {\bf 0}_{N_{v_1},1}  & \vdots  &  {\bf 0}_{N_{v_1},1} \\
 {\bf 0}_{1,N_{v_1}} & \hdots  & {\bf 0}_{1,N_{v_1}} & {\bf 0}_{1,N_{v_1}} & 0 &   -ik & 0  \\
 -\Delta v_1\Delta v_2 (\vec{v_2})_1 {\bf 1} & -\Delta v_1\Delta v_2 (\vec{v_2})_2 {\bf 1} & \hdots & -\Delta v_1\Delta v_2 (\vec{v_2})_{N_{vy}} {\bf 1} & -ik & 0 & 0\\
  -\Delta v_1\Delta v_2 \vec{v_1} & -\Delta v_1\Delta v_2 \vec{v_1} & \hdots  & -\Delta v_1\Delta v_2 \vec{v_1} & 0& 0& 0 
\end{array} 
 \right) 
  \end{eqnarray}  
 where we denote $\vec{v_1}$ the vector with components $(\vec{v_1})_{j_1}=v_{1, \min} + j_1 \Delta v_1$ and 
 $\vec{v_2}$ the vector with components $(\vec{v_2})_{j_2}=v_{2, \min} + j_2 \Delta v_2$. Moreover, 
diag$(\vec{v_1})$ denotes the diagonal matrix whose diagonal is composed of $\vec{v_1}$, ${\bf 1}\in {\mathbb{R}}^{N_{v_1}}$ denotes the vector with components $1$ and 
${\bf 0}_{A,B}$ is a matrix with $A$ lines and $B$ columns with zeros. 
 The size of the matrix $L$ is $(N_{v_1}N_{v_2} + 3)$ and in spite 
 of its size, one can see that $L$ is sparse. 
 
 \medskip 
 
 A key point is to compute $\exp( L\Delta t)$ to design an exponential scheme approximating \eqref{ULU}. Denoting $U^n\approx U(t^n), t^n=n\Delta t, \Delta t>0$, the first order Lawson scheme is  
 $$
 U^{n+1} = \exp( L\Delta t)U^n + \Delta t \exp( L\Delta t) N(U^n).
 $$  
 The following proposition gives an explicit expression of $\exp(-\Delta t L)$.
 \begin{pro}
  The exponential of the matrix $L$ given by \eqref{L_fourier} 
  is given by 
 $$
 \left( 
 \begin{array}{lllllllccc}
e^{-ik \Delta t \mbox{diag}(\vec{v_{1}}) }& {\bf 0}_{N_{v_1},N_{v_1}} & {\bf 0}_{N_{v_1},N_{v_1}}\hdots &   {\bf 0}_{N_{v_1},N_{v_1}} &  {\bf 0}_{N_{v_1},1} & {\bf 0}_{N_{v_1},1} & {\bf 0}_{N_{v_1},1}\\
 {\bf 0}_{N_{v_1},N_{v_1}} &e^{-ik \Delta t \mbox{diag}(\vec{v_{1}}) } &  {\bf 0}_{N_{v_1},N_{v_1}}\hdots &   {\bf 0}_{N_{v_1},N_{v_1}} & \vdots &  \vdots & \vdots\\
\vdots & \ddots & \ddots  &  {\bf 0}_{N_{v_1},N_{v_1}}  &  \vdots & \vdots & \vdots\\
 {\bf 0}_{N_{v_1},N_{v_1}} & \hdots  &   {\bf 0}_{N_{v_1},N_{v_1}} & e^{-ik \Delta t \mbox{diag}(\vec{v_{1}}) }  &  {\bf 0}_{N_{v_1},1}  &  {\bf 0}_{N_{v_1},1}  &  {\bf 0}_{N_{v_1},1} \\
 {}_e{\cal B}_{k,1} & {}_e{\cal B}_{k,2}  & \hdots & {}_e{\cal B}_{k,N_{v_2}} &  \cos(k \Delta t) & -i\sin(k \Delta t)& 0   \\
{}_e{\cal E}_{2,k,1} & {}_e{\cal E}_{2,k,2} & \hdots  & {}_e{\cal E}_{2,k,N_{v_2}} & -i\sin(k\Delta t)  & \cos(k\Delta t)& 0  \\
{}_e{\cal E}_{1,k}& {}_e{\cal E}_{1,k} & \hdots & {}_e{\cal E}_{1,k} &  0& 0 & 1
\end{array} 
 \right)
$$
where ${}_e{\cal B}_{k,j_2},{}_e{\cal E}_{2,k,j_2}, {}_e{\cal E}_{1,k} 
\in\mathbb{M}_{1,N_{v_1}}(\mathbb{C})$ for $j_2=1, \dots, N_{v_2}$ are given by 
\begin{eqnarray*}
{}_e{\cal B}_{k,j_2}&=& \frac{i\Delta v_1 \Delta v_2}{k}v_{2, j_2} \vec{\alpha}, \;\; 
{}_e{\cal E}_{2,k,j_2}=  \frac{i\Delta v_1 \Delta v_2}{k}v_{2, j_2} \vec{\beta}, \;\; 
{}_e{\cal E}_{1,k} = \frac{i\Delta v_1 \Delta v_2}{k}({\bf 1}-  e^{-ik \vec{v_{1}}}), 
\end{eqnarray*}
where the vectors $\vec{\alpha}, \vec{\beta}\in\mathbb{M}_{1,N_{v_1}}(\mathbb{C})$ whose components are given by 
\begin{eqnarray*}
{\alpha}_{j_1} &=&  -\frac{ e^{-ik\Delta t} }{2(1-v_{1, j_1})} - \frac{ e^{ik\Delta t} }{2(1+v_{1, j_1})} + \frac{ e^{-ik \Delta t  v_{1, j_1}} }{(1-v_{1, j_1}^2)},\nonumber\\
{\beta}_{j_1} &=& -\frac{e^{-ik\Delta t} }{2(1-v_{1, j_1})} + \frac{e^{ik\Delta t} }{2(1+v_{1, j_1})}+ \frac{v_{1, j_1} e^{-ik \Delta t  v_{1, j_1}} }{(1-v_{1, j_1}^2)}.
\end{eqnarray*}
\end{pro}
 
 \begin{proof}
To compute $\exp(\Delta tL)$, we will solve exactly the linear part of the Vlasov-Maxwell system. 
 First of all, we observe that the components associated to $\hat{f}$ is diagonal and can be solved independently so that the $N_{v_1}\times N_{v_2}$ top left block of $\exp(\Delta tL)$ is diagonal and is equal 
 to $e^{-ik \Delta t \vec{v_1}}$ ($N_{v_2}$ times). Second, 
 the $3\times 3$ right bottom block corresponds to the homogeneous Maxwell equations and its exponential can be computed as 
 $$
  \exp \left( 
 \begin{array}{lll}
 0& -ik\Delta t & 0\\
 -ik\Delta t & 0 & 0 \\
0 & 0 & 0
 \end{array}
 \right) 
 =
 \left(
 \begin{array}{lll}
 \cos(k \Delta t) & i\sin(k \Delta t)& 0\\
i\sin(k\Delta t)  & \cos(k\Delta t)& 0\\
 0& 0 & 1
 \end{array}
 \right).  
 $$
It remains to compute the three last lines corresponding to the coupling between the Vlasov and Maxwell equations. 
 \paragraph{Computation of ${}_e{\cal E}_{1,k}$: solve $\hat{E}_{1,k}$\\}
First, we have for $\hat{f}_{k,j_1, j_2}(t)$  
$$
\hat{f}_{k,j_1,j_2}(t) = e^{-ik v_{1, j_1} (t-t^n)}\hat{f}_{k,j_1,j_2}(t^{n}), 
$$
which enables to compute $\hat{E}_{1,k}(t^{n+1})$
\begin{eqnarray*}
 \hat{E}_{1,k}(t^{n+1})&=&\hat{E}_{1,k}(t^{n}) -  \sum_{j_1, j_2} \int_{t^n}^{t^{n+1}} e^{-ik v_{1, j_1} (t-t^n)} dtv_{1, j_1}{f}_{k,j_1,j_2}(t^{n}) \Delta v_1 \Delta  v_2\\
 &=& \hat{E}_{1,k}(t^{n}) + \frac{i\Delta v_1 \Delta v_2}{k}\sum_{j_1, j_2}(1-  e^{-ik v_{1, j_1} \Delta t} )\hat{f}_{k,j_1,j_2}(t^{n}),  
\end{eqnarray*}
from which we deduce the expression of ${}_e{\cal E}_{1,k}$. 

 \paragraph{Computation of ${}_e{\cal B}_{k}, {}_e{\cal E}_{2,k}$: solve $\hat{B}_{k}, \hat{E}_{2,k}$\\}
Next, we focus on the calculation of $\hat{E}_{2,k}(t^{n+1})$ and $\hat{B}_{k}(t^{n+1})$. We write down the equations  
\begin{eqnarray*}
\frac{d}{dt} \hat{E}_{2,k}(t)&=& - ik \hat{B}_k(t) - \sum_{j_1, j_2} e^{-ik v_{1, j_1} (t-t^n)} v_{2, j_1}{f}_{k,j_1,j_2}(t^{n}) \Delta v_1 \Delta  v_2 \\
\frac{d}{dt} \hat{B}_k(t)&=&  - ik  \hat{E}_{2,k}(t) 
 \end{eqnarray*}
which can be rewritten as $\frac{dU}{dt} = MU + R$ with $U(t)=(\hat{E}_{2,k}(t), \hat{B}_k(t))$ and 
$$
M=
\left( 
\begin{array}{ll}
0& -ik \\ 
 -ik & 0
\end{array}
\right) \mbox{ and } 
R=
\left( 
\begin{array}{ll}
R_1(t) \\ 
 0
\end{array}
\right), \mbox{ with } R_1(t)=- \sum_{j_1, j_2} e^{-ik v_{1, j_1} (t-t^n)} v_{2, j_1}\hat{f}_{k,j_1,j_2}(t^{n}) \Delta v_1 \Delta  v_2.   
 $$
Thus, one can write the variation of constant formula 
\begin{equation}
\label{duhamel}
U(t^{n+1}) = e^{M\Delta t} U(t^n) + \int_{t^n}^{t^{n+1}} e^{-M(t-t^{n+1})} R(t) dt.  
\end{equation}
First, $e^{M\Delta t}$ reads as 
$$
 e^{M\Delta t} = \exp \left( 
 \begin{array}{ll}
 0 & -ik\Delta t\\
- ik\Delta t & 0
 \end{array}
 \right) 
 =
 \left(
 \begin{array}{ll}
 \cos(kt)& -i\sin(kt)\\
-i\sin(kt)  & \cos(kt)
 \end{array}
 \right).  
$$
Second, one has to compute the integral term in \eqref{duhamel} 
\begin{eqnarray}
\int_{t^n}^{t^{n+1}} e^{-M(t-t^{n+1})} R(t) dt &=& 
 \left(
 \begin{array}{ll}
\int_{t^n}^{t^{n+1}} \cos(k(t-t^{n+1})) R_1(t) dt\\
\int_{t^n}^{t^{n+1}} i \sin(k(t-t^{n+1})) R_1(t)dt \\
  \end{array}
 \right)\nonumber\\
 &=& -
 \left(
  \begin{array}{ll}
  \sum_{j_1, j_2} {\cal I}_1 v_{2, j_1}\hat{f}^{n}_{k,j_1,j_2} \Delta v_1 \Delta  v_2  \\
  \sum_{j_1, j_2} {\cal I}_2 v_{2, j_1}\hat{f}^{n}_{k,j_1,j_2} \Delta v_1 \Delta  v_2 
  \end{array}
 \right)
 \label{expM}
\end{eqnarray} 
where ${\cal I}_1$ and  ${\cal I}_2$ are given by 
\begin{eqnarray*}
{\cal I}_1&=&\int_{t^n}^{t^{n+1}}\Big[ \cos(k(t-t^{n+1}))  e^{-ik v_{1, j_1} (t-t^n)}\Big]dt = \frac{i e^{-ik\Delta t} }{2k(1-v_{1, j_1})} - \frac{i e^{ik\Delta t} }{2k(1+v_{1, j_1})} - \frac{i v_{1, j_1} e^{-ik \Delta t  v_{1, j_1}} }{k(1-v_{1, j_1}^2)}, \nonumber\\
{\cal I}_2&=& \int_{t^n}^{t^{n+1}}\Big[i \sin(k(t-t^{n+1})) e^{-ik v_{1, j_1} (t-t^n)} \Big]dt=\frac{i e^{-ik\Delta t} }{2k(1-v_{1, j_1})} + \frac{i e^{ik\Delta t} }{2k(1+v_{1, j_1})} - \frac{i  e^{-ik \Delta t  v_{1, j_1}} }{k(1-v_{1, j_1}^2)}.
\end{eqnarray*} 
Inserting these calculations in \eqref{duhamel} leads to the following expression for $\hat{E}_{2,k}(t)$
\begin{eqnarray*}
 \hat{E}_{2,k}(t^{n+1})&=&  \cos(k\Delta t) \hat{E}_{2,k}(t^n) - i\sin(k\Delta t) \hat{B}_{k}(t^n) + \frac{i\Delta v_1 \Delta v_2}{k} \sum_{j_1, j_2} v_{2, j_2} \beta_{j_1} \hat{f}_{k,j_1, j_2}(t^n),   \nonumber\\
 \hat{B}_{k}(t^{n+1}) &=&-i\sin(k\Delta t) E_{2,k}(t^n) +\cos(k\Delta t) B_{k}(t^n) + \frac{i\Delta v_1 \Delta v_2}{k} \sum_{j_1, j_2} v_{2, j_2} \alpha_{j_1} \hat{f}_{k,j_1, j_2}(t^n), 
\end{eqnarray*}
where $\vec{\beta}=[\beta_1, \beta_2, \dots, \beta_{N_{v_1}}]\in \mathbb{C}^{N_{v_1}}$ and 
$\vec{\alpha}=[\alpha_1, \alpha_2, \dots, \alpha_{N_{v_1}}]\in \mathbb{C}^{N_{v_1}}$ are given by 
\begin{eqnarray*}
\beta_{j_1}&=& -\frac{e^{-ik\Delta t} }{2(1-v_{1, j_1})} + \frac{e^{ik\Delta t} }{2(1+v_{1, j_1})}+ \frac{v_{1, j_1} e^{-ik \Delta t  v_{1, j_1}} }{(1-v_{1, j_1}^2)}, \nonumber\\
\alpha_{j_1}&=&  -\frac{ e^{-ik\Delta t} }{2(1-v_{1, j_1})} - \frac{ e^{ik\Delta t} }{2(1+v_{1, j_1})} + \frac{ e^{-ik \Delta t  v_{1, j_1}} }{(1-v_{1, j_1}^2)}. 
\end{eqnarray*}

We conclude by writing the vectors ${}_e{\cal B}_{k},{}_e{\cal E}_{2,k}$ corresponding to $\hat{B}_k$ and $\hat{E}_{2,k}$
$$
{}_e{\cal B}_{k,j_2} =  \frac{i\Delta v_1 \Delta v_2}{k} v_{2, j_2} \vec{\alpha}, \;\; {}_e{\cal E}_{2,k,j_2} =  \frac{i\Delta v_1 \Delta v_2}{k} v_{2, j_2} \vec{\beta}. 
$$
\end{proof}

\newpage

\bibliographystyle{abbrv}
\bibliography{exponentialDG}

\end{document}